\documentclass[11pt, a4paper]{article}
\usepackage{t1enc}
\usepackage[latin1]{inputenc}
\usepackage[english]{babel}
\usepackage{amsmath,amsthm}
\usepackage{mathrsfs}
\usepackage{stmaryrd, mathtools, amssymb, dsfont, mathrsfs, bm}
\usepackage{amsfonts}
\usepackage{latexsym}
\usepackage[dvips]{graphicx}
\usepackage{float}
\usepackage[natural]{xcolor}
\usepackage{algorithm}
\usepackage{algorithmic}
\usepackage{enumerate,enumitem}
\usepackage{multirow}
\usepackage{xcolor}
\usepackage[colorlinks,linkcolor=blue]{hyperref}
\usepackage[noabbrev,capitalise]{cleveref}
\usepackage{tasks}
\usepackage{lineno}
\usepackage{setspace}
\usepackage{fullpage}
\usepackage[title]{appendix}
\usepackage{todonotes}
\usepackage{amssymb}
\usepackage{array}
\usepackage{subcaption}

\usepackage{listings}
\usepackage{bbm}
\usepackage{tikz}

\usepackage{caption}
\usepackage{url}
\usepackage{fullpage}
\usepackage{hyperref}
\usepackage[margin=2.3cm]{geometry}
\usepackage{enumitem,verbatim}

\newtheorem{theorem}{Theorem}[section]
\newtheorem{cl}[theorem]{Claim}
\newtheorem{claim}[theorem]{Claim}
\newtheorem{lemma}[theorem]{Lemma}
\newtheorem{conj}[theorem]{Conjecture}

\newtheorem{proposition}[theorem]{Proposition}

\newtheorem{remark}{Remark}

\usepackage{enumitem}

\newcommand\nlabel{\upshape({\itshape\arabic*\,\/})}

\newcommand{\bE}{\mathbb{E}}

\newcommand{\cF}{\mathcal{F}}
\newcommand{\cL}{\mathcal{L}}
\newcommand{\cP}{\mathcal{P}}
\newcommand{\bS}{\mathbf{S}}
\newcommand{\bX}{\mathbf{X}}
\newcommand{\esp}[1]{\mathbb{E}\left[ #1 \right]}
\newcommand{\pr}[1]{\mathbb{P}\left[ #1 \right]}
\numberwithin{equation}{section}

\theoremstyle{definition}

\newtheorem{definition}[theorem]{Definition}

\usepackage{todonotes}

\title{An exact Ore-degree condition for Hamilton cycles in oriented graphs}

\author{Yulin Chang\thanks{Mathematical Institute, Ocean University of China. Email: {\tt ylchang@ouc.edu.cn}. Supported by Natural Science Foundation of China (12201352) and Natural Science Foundation of Shandong Province (ZR2022QA083).}\and
Yangyang Cheng\thanks{Mathematical Institute, University of Oxford, Oxford, UK. Email: {\tt yangyang.cheng@maths.ox.ac.uk}. Supported by the PhD studentship of ERC Advanced Grant (883810).}
\and
Tianjiao Dai \thanks{School of Mathematics, Shandong University, Jinan, 250100, China. Email: {\tt tianjiao.dai@sdu.edu.cn}. Supported by Natural Science Foundation of Shandong Province(ZR2024QA174).}
\and
Qiancheng Ouyang \thanks{Data Science Institute, Shandong University, Jinan, 250100, China. Email: {\tt oyqc@sdu.edu.cn}. Supported by Natural Science Foundation of Shandong Province(ZR2024QA107).}
\and
Guanghui Wang \thanks{School of Mathematics, Shandong University, Jinan, 250100, China. Email: {\tt ghwang@sdu.edu.cn}. Supported by the Natural Science Foundation of China (12231018).}
}
\date{}
\begin{document}
\maketitle
\begin{abstract}

An oriented graph is a digraph that contains no 2-cycles, i.e., there is at most one arc between any two vertices. We show that every oriented graph $G$ of sufficiently large order $n$ with $\deg^+(x) +\deg^{-}(y)\geq (3n-3)/4$ whenever $G$ does not have an edge from $x$ to $y$ contains a Hamilton cycle. This is best possible and solves a problem of  K\"uhn and Osthus from 2012. Our result generalizes the result of Keevash, K\"uhn, and Osthus and improves the asymptotic bound obtained by Kelly, K\"uhn, and Osthus. \\

\end{abstract}
\maketitle
\section{Introduction}
Given a fixed graph $H$, let $G$ be a graph of the same order as $H$. 
A fundamental question in graph theory is under what conditions $G$ contains a subgraph isomorphic to $H$. Such problems have received significant attention throughout the history of graph theory. For most graphs $H$, determining whether $G$ contains $H$ as a subgraph is very challenging. For instance, when $H$ is a Hamilton cycle, this decision problem has been shown to be NP-complete and is, in fact, one of the prototypes of NP-complete problems in complexity theory. Thus, it is interesting to seek sufficient conditions that guarantee a Hamilton cycle or even a more general $H$.

The famous Dirac's theorem \cite{dirac1952some} asserts that every graph $G$ of order $n\ge3$ with minimum degree $\delta(G) \geq n/2$ contains a Hamilton cycle. Ore \cite{ore1960note} generalizes Dirac's result by proving that a graph $G$ of order $n\ge3$ contains a Hamilton cycle, if for every pair of vertices $x, y\in V(G)$ with $xy \notin E(G)$, $\deg(x) + \deg(y) \geq n$. A further generalization, due to Chv\'{a}tal \cite{chvatal1972hamilton}, states that if the degree sequence of a graph $G$ is $d_1 \leq d_2 \leq \cdots \leq d_n$, and if $n \geq 3$ and $d_i \geq i+1$ or $d_{n-i} \geq n-i$ for all $i > n/2$, then $G$ contains a Hamilton cycle.
The degree sequence condition given by Chv\'{a}tal is best possible in the sense that, for any degree sequence $d_1 \leq d_2 \leq \cdots \leq d_n$ that violates Chv\'{a}tal's condition, one can always construct a graph $H$ with degree sequence $d'_1 \leq d'_2 \leq \cdots \leq d'_n$ such that $d'_i \geq d_i$ for all $i \in [n]$, which does not contain a Hamilton cycle.

For a digraph $G$ and a vertex $v \in V(G)$, let $\deg^+(v)$ and $\deg^-(v)$ denote the outdegree and indegree of $v$, respectively. Define $\delta^+(G) \coloneqq \min \{\deg^+(v) \colon v \in V(G)\}$ and $\delta^-(G) \coloneqq \min \{\deg^-(v)  \colon v \in V(G)\}$. Let $\delta^0(G) \coloneqq \min\{\delta^+(G), \delta^-(G)\}$. When referring to paths and cycles in digraphs we always mean that they are directed without mentioning this explicitly. The Hamilton cycle problem was generalized to digraphs by Ghouila-Houri \cite{ghouilahouri1960condition}:
\begin{theorem}[Ghouila-Houri \cite{ghouilahouri1960condition}]
Every strongly connected digraph $G$ on $n$ vertices with $\delta^+(G)+\delta^-(G)\geq n$ contains a Hamilton cycle. In particular, every digraph with $\delta^0(G)\geq n/2$ contains a Hamilton cycle.
\end{theorem}
 For Ore-type conditions, Woodall \cite{woodall1972sufficient} extended Ore's theorem to the digraph setting:
\begin{theorem}[Woodall \cite{woodall1972sufficient}]
Every strongly connected digraph $G$ on $n$ vertices with $\deg^+(x)+\deg^-(y)\geq n$ for every pair $x\neq y$ with $xy\notin E(G)$ contains a Hamilton cycle.
\end{theorem}

 However, proving a Chv\'{a}tal-type condition for digraphs is significantly more challenging. A long-standing conjecture by Nash-Williams proposes a Chv\'{a}tal-type condition for Hamilton cycles (see \cite{nashst}):
 \begin{conj}[Nash-Williams \cite{nashst}]\label{nash-williams}
Suppose that $G$ is a strongly connected digraph on $n\geq 3$ vertices such that for all $i<n/2$,
\begin{itemize}
  \item $d^+_i\geq i+1$ or $d^-_{n-i}\geq n-i$;
  \item $d^-_i\geq i+1$ or $d^+_{n-i}\geq n-i$,
\end{itemize}
then $G$ contains a Hamilton cycle.
\end{conj}

 In \cite{christofides2010semiexact}, Christofides, Keevash, K\"{u}hn, and Osthus proved an asymptotic version of the Nash-Williams conjecture, but the full conjecture remains open.

An oriented graph is a digraph that contains no $2$-cycles, i.e., there is at most one arc between any two vertices. Keevash, K\"{u}hn, and Osthus \cite{keevash2009exact} generalized Dirac's theorem to oriented graphs, proving that there exists an $n_0$ such that for all $n \geq n_0$, every oriented graph $G$ of order $n$ with $\delta^0(G) \geq \frac{3n-4}{8}$ contains a Hamilton cycle. Kelly, K\"{u}hn, and Osthus \cite{kelly2008dirac} established an Ore-degree condition for oriented graphs, showing that for every $\alpha > 0$, there exists an integer $n_0 = n_0(\alpha)$ such that every oriented graph $G$ of order $n \geq n_0$ satisfying $\deg^+(x) + \deg^-(y) \geq (3/4 + \alpha)n$ for each $x, y$ with $xy \notin E(G)$ contains a Hamilton cycle.
K\"{u}hn and Osthus~\cite{kuhn2012survey} presented that it would be interesting to obtain a tight version of this result. The result of our work is to give a full solution to this problem when $n$ is sufficiently large, which generalizes the result in \cite{keevash2009exact}.
\begin{theorem}\label{main}
There exists an $n_0>0$ such that for all $n\geq n_0$, every oriented graph $G$ of order $n$ satisfying $\deg^+(x) + \deg^-(y) \geq  (3n-3)/4$ whenever $xy \notin E(G)$ contains a Hamilton cycle.  
\end{theorem}

The proof of Theorem~\ref{main} is separated into the following two theorems, depending on whether a graph is close to an extremal construction. An oriented graph $G$ is $\eta$-extremal for some $\eta>0$ if by adding or deleting at most $O(\eta) n^2$ edges of $G$, we obtain an extremal construction as defined in Section \ref{sec:extremal} (see Definition \ref{ex-defi}).

\begin{theorem}[Non-extremal case]\label{thm:non-extremal}
    Let $0<1/n\ll \eta\ll 1$. Let $G$ be an oriented graph of order $n$ satisfying $\deg^+(x) + \deg^-(y) \geq (3n-3)/4$ whenever $xy \notin E(G)$. Then one of the following holds: 
    \begin{enumerate}[label={\rm(\roman*)}]
        \item $G$ contains a Hamilton cycle;
        \item $G$ is $\eta$-extremal. 
    \end{enumerate}
\end{theorem}
\begin{theorem}[Extremal case]\label{thm:extremal}
    Let $0<1/n\ll \eta\ll 1$. Let $G$ be an oriented graph of order $n$ satisfying $\deg^+(x) + \deg^-(y) \geq (3n-3)/4$ whenever $xy \notin E(G)$. Suppose $G$ is $\eta$-extremal. Then $G$ contains a Hamilton cycle.
\end{theorem}

\subsection{Organization of the paper}

The rest of the paper is organized as follows.
In Section~\ref{s0}, we discuss the strategy of our proof.
In Section~\ref{s2}, we introduce the necessary notation, present fundamental propositions, and describe the extremal examples demonstrating that Theorem~\ref{main} is essentially best possible.
In Section~\ref{s3} we introduce the main
tools and lemmas for the proof of Theorem~\ref{main}. 
In Section~\ref{proof-non ex} we prove~\Cref{thm:non-extremal}. The key tools of the absorption method will be proved in Section~\ref{sec:abs}. In Section~\ref{s5} we show our covering lemma. Finally, we give the proof of Theorem~\ref{thm:extremal} in Section \ref{section-extremal}.

\subsection{Proof Strategy}\label{s0}
Our proof uses the absorption technique first introduced by R\"{o}dl, Ruci\'{n}ski, and Szemer\'{e}di~\cite{RRS2006}, as well as stability analysis for the extremal case. 
We first prove that for every non-$\eta$-extremal oriented graph $G$ of order $n$ such that $\deg^+(x) + \deg^-(y) \geq (3n-3)/4$ for every pair of vertices $(x, y)$ with $xy \notin E(G)$ contains a Hamilton cycle, where $\eta$ is very small and $n$ is sufficiently large compared to $1/\eta$. In this step, we adapt the idea of the absorption technique by building a small absorption structure $\mathcal{A}$ that can ``absorb'' vertices. Roughly speaking, $\mathcal{A}$ consists of disjoint sets of $1$ or $2$ edges, designed for different purposes, such that we can insert a small set of vertices into $\mathcal{A}$ to transform it into a collection of disjoint directed paths of slightly larger size, while preserving the number of directed paths. The total size of $\mathcal{A}$ is small compared to $n$, and when we insert the vertices, we do not change the endpoints of the directed paths in $\mathcal{A}$.

Next, we prove our connecting lemma that enables us to join the disjoint paths in $\mathcal{A}$ into a single directed path $P_{abs}$. The path $P_{abs}$ remains small compared to $n$, and for every vertex set $U$ that is relatively small compared to $|V(P_{abs})|$, we can absorb $U$ into $P_{abs}$, thereby obtaining a larger directed path $P'$ with $V(P') = V(P_{abs}) \cup U$.

In addition to the absorbing path $P_{abs}$, we also construct a \emph{reservoir set} $R$, which consists of a collection of disjoint short directed paths and has size much smaller than $|V(P_{abs})|$. The set $R$ has the property that for every pair of vertices $(u,v)$, there exist many disjoint directed paths $P_1, \ldots, P_t$ of bounded length within $R$ such that $uP_iv$ is a directed path for each $1 \leq i \leq t$. This allows us to use disjoint short paths from $R$ to connect a collection of disjoint paths into a longer directed path or a directed cycle.

Once the absorbing path $P_{abs}$ and the reservoir set $R$ have been constructed, the remaining task is to cover the vertices in $V(G) \setminus (V(P_{abs}) \cup V(R))$ by disjoint directed paths such that their total number is not too large. In this step, let $G' = G - (V(P_{abs}) \cup V(R))$. We apply the degree form of the Diregularity Lemma to $G'$ and obtain a reduced oriented graph $G^{\text{red}}$. Note that $G^{\text{red}}$ has a slightly weaker degree condition than $G$, but still remains close.
We then show that $G^{\text{red}}$ contains a $1$-factor; otherwise, $G$ would be $\eta$-extremal. Let $F$ be the $1$-factor we find. The final step is to use the Blow-up Lemma to find disjoint directed paths corresponding to $F$, such that the total number of paths is relatively small. We then use the reservoir set $R$ to connect these directed paths and $P_{abs}$ into a large directed cycle $C$ that is almost Hamiltonian. Note that $R$ may still contain some remaining vertices, and there is a relatively small vertex subset $V_0$ due to the Diregularity Lemma. The final step is to absorb these vertices into $P_{abs}$, thereby transforming $C$ into a directed Hamilton cycle.

The covering part of our proof is standard. Here, we introduce some novel techniques used in constructing the absorbing path $P_{abs}$. Specifically, we develop a \emph{double-step absorption} method.
Formally, the absorbing structure $\mathcal{A}$ contains two kinds of families of directed paths, $\mathcal{A}_1$ and $\mathcal{A}_2$, constructed using probabilistic methods. Here, $\mathcal{A}_1$ is a set of disjoint edges designed to absorb vertices (or pairs) that are \emph{strongly absorbable}, while $\mathcal{A}_2$ is used for those that are \emph{weakly absorbable}.
We will define these two absorbers in Section~\ref{subsec:5.3}. 
If a vertex $v$ (or a pair $(u,v)$, where $u,v$ are the endpoints of a directed path $P$ with $V(P) \cap V(P_{abs}) = \emptyset$) is strongly absorbable, then the path $P_{abs}$ can absorb it directly by replacing some edge $wz$ in $\mathcal{A}_1$ with $wvz$ (or with $wPz$), resulting in a new path $P'$ such that $V(P') = V(P_{abs}) \cup \{v\}$ (or $V(P') = V(P_{abs}) \cup V(P)$). This process uses only edges from $\mathcal{A}_1$, and each absorption consumes exactly one edge from $\mathcal{A}_1$.

However, our Ore-degree condition is not sufficiently strong to guarantee that all vertices or pairs are strongly absorbable. But we can show that if some vertex is not strongly absorbable, then it must be weakly absorbable.
Using this fact, if a vertex $v$ is weakly absorbable, we can absorb it into $P_{abs}$ by replacing a directed path $ww'Pz'z$ from $\mathcal{A}_2$ with $wvz$, obtaining a new path $P_1$ such that $
V(P_1) = V(P_{abs}) \cup \{v\} \setminus (\{w', z'\}\cup \{V(P)\})$, where the pair $(w', z')$ is strongly absorbable by definition. Therefore, we can absorb the path $w'Pz'$ into $P_1$ again by replacing one edge $uv$ from $\mathcal{A}_1$ by $uw'Pz'v$, resulting in a directed path $P_2$ with
$V(P_2) = V(P_1) \cup \{w', z'\}\cup V(P).$

The absorbing process described here can be repeated many times, allowing us to absorb a linear number of vertices, provided that their total number is relatively small compared to the size of $P_{abs}$. According to the strategy outlined earlier, this eventually yields a directed Hamilton cycle once we obtain a directed cycle containing $P_{abs}$ with order close to $n$.

To finish the proof, it remains to consider the case when $G$ is $\eta$-extremal. This means that $G$ is close to an extremal construction. The proof in this case requires detailed structural analysis, see Section~\ref{section-extremal}.

\section{Preliminaries }\label{s2}

\subsection{Notation}

Given two vertices $x$ and $y$ of a digraph $G$, we write $xy$ for the edge directed from $x$ to $y$. We write $N_G^+(x)$ for the outneighborhood of a vertex $x$ and $N_G^-(x)$ for the inneighborhood of a vertex $x$. We write $N_G(x) \coloneqq N_G^+(x) \cup N_G^-(x)$ for the neighborhood of a vertex $x$. A digraph $G$ satisfies \emph{Ore-degree condition} if for every pair of vertices $x$ and $y$ of $V(G)$ with $xy\notin E(G)$, $\deg^+(x)+\deg^-(y)\ge(3n-3)/4$.

Given a set $X\subseteq V(G)$, we write $N^+_G(X)$ for the set of all outneighbors
of vertices in $X$. So $N^-_G(X)$ is defined similarly.
The directed subgraph of $G$ induced by $X$ is denoted by $G[X]$. We write $E(X)$ for the set of edges of $G[X]$ and put $e(X) \coloneqq |E(X)|$. 
Given two disjoint subsets $X$ and $Y$ of $V(G)$, an $X$--$Y$ edge is an edge $xy$, where $x\in X$ and $y\in Y$. We write $E(X,Y)$ for the set of all edges $xy$ with $x\in X$ and $y\in Y$, and put $e(X,Y) \coloneqq |E(X,Y)|$. We denote $G[X,Y]$ to be the subgraph of $G$ with vertex set $X\cup Y$ and edge set $E(X,Y)$. We omit the subscripts $G$ when there is
no danger of confusion.

Given two vertices $x$ and $y$, an \emph{$(x,y)$-path} is a directed path which joins $x$ to $y$.
A \emph{$1$-factor} of a digraph $G$ is a collection of disjoint
cycles which cover all the vertices of $G$.
The underlying graph of a digraph $G$ is the graph obtained from $G$ by ignoring
the directions of its edges.
Throughout the paper, we omit floor and ceiling functions unless it is necessary.
We write $\alpha\ll \beta\ll \gamma$ to mean that we can choose positive constants $\alpha,\beta,\gamma$ from right to left.
More precisely, there are two monotone functions $f$ and $g$ such that, given $\gamma$, whenever $\beta\le f(\gamma)$ and $\alpha\le g(\beta)$, the subsequent statements hold.
Hierarchies with more constants are defined analogously.
Implicitly, we assume that all constants appearing in a hierarchy are positive.

\subsection{Extremal examples}\label{sec:extremal}
In this section, we show that the bound in Theorem~\ref{main} is best possible for every $n\in\mathbb{N}$.

\begin{proposition}\label{prop:ex}
For any $n\ge 3$ there are infinitely many oriented graphs $G$ on $n$ vertices 
satisfying $\deg^+(x) + \deg^-(y)=\lceil (3n-3)/4\rceil-1$ for some pair of vertices $(x,y)$ with $x,y\in V(G)$ and $xy \notin E(G)$ do not contain a Hamilton cycle.
\end{proposition}

 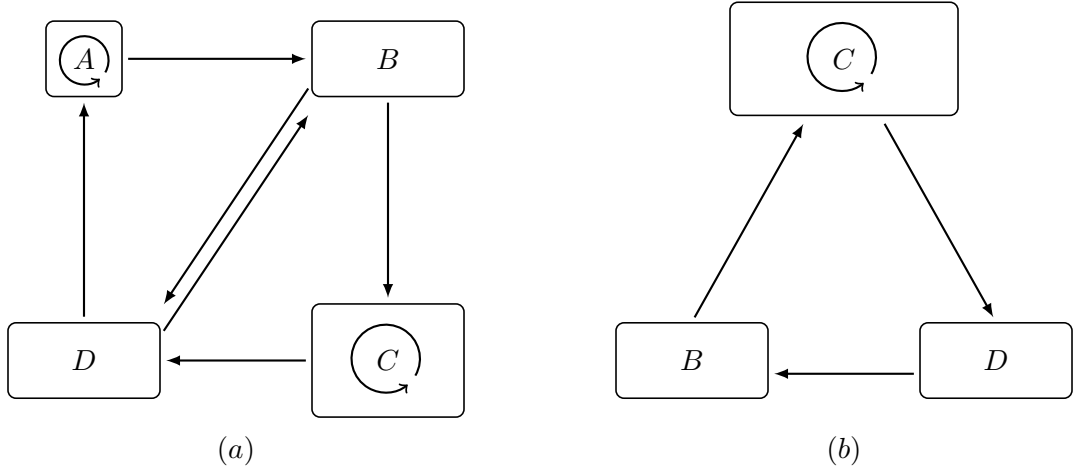
\begin{figure}
     \centering
\begin{tikzpicture}[
    box/.style={
        rectangle, 
        rounded corners=3pt, 
        draw=black, 
        line width=0.6pt,
        minimum width=2cm,
        minimum height=1cm,
        align=center
    },
    arrowfwd/.style={
        ->,
        >=latex,
        line width=0.8pt,
        shorten >=2pt,
        shorten <=2pt
    },
    arrowbwd/.style={
        <-,
        >=latex,
        line width=0.8pt,
        shorten >=2pt,
        shorten <=2pt
    }
]

\node[box,minimum width=1cm] (A) at (0,2) {$A$};
    \begin{tikzpicture}[remember picture, overlay]
      \draw[->, line width=0.8pt] (0.28,1.82) arc (-30:310:0.32cm);
    \end{tikzpicture}
\node[box] (B) at (4,2) {$B$};
\node[box,minimum height=1.5cm] (C) at (4,-2) {$C$};
       \draw[->, line width=0.8pt] (4.36,-2.2) arc (-30:310:0.45cm);
\node[box] (D) at (0,-2) {$D$};

\draw[arrowfwd] (A.east) -- (B.west);
\draw[arrowfwd] (B.south) -- (C.north);
\draw[arrowfwd] (C.west) -- (D.east);
\draw[arrowfwd] (D.north) -- (A.south);

\draw[arrowfwd,black] ([yshift=5pt]B.south west) -- ([yshift=5pt]D.north east)
    node[midway,right=1mm,font=\scriptsize] {};
\draw[arrowbwd,black] ([yshift=-5pt]B.south west) -- ([yshift=-5pt]D.north east)
    node[midway,left=1mm,font=\scriptsize] {};


\node[box,minimum height=1.5cm,minimum width=3cm] (C) at (10,2) {$C$};
\node[box,minimum height=1cm] (D) at (12,-2) {$D$};
       \draw[->, line width=0.8pt] (10.36,1.8) arc (-30:310:0.45cm);
\node[box,minimum height=1cm] (B) at (8,-2) {$B$};
\node at (10, -3.2) {$(b)$};
\node at (2, -3.2) {$(a)$};

\draw[arrowfwd] (B.north) -- (9.5,1.2);
\draw[arrowfwd] (10.5,1.2) -- (D.north);

 \draw[arrowbwd,black] ([yshift=-5pt]B.east) -- ([yshift=-5pt]D.west)
   node[midway,left=1mm,font=\scriptsize] {};
\end{tikzpicture}

    \caption{The oriented graphs in the proof of Proposition~\ref{prop:ex}. Note that $|A|$ may be any integer satisfying $0\le |A|\le |C|$. When $|A|=0$, we have $e(B,D)=0$, and the extremal example is illustrated in~$(b)$.}\label{fig:extremal}
\end{figure}

\begin{proof}
    We construct $G$ as follows. It has $n$ vertices partitioned into four parts $A,B,C,D$ with $|B|>|D|$. Each of $A$ and $C$ spans a tournament, $B$ and $D$ are joined by a bipartite tournament, and we add all edges from $A$ to $B$, from $B$ to $C$, from $C$ to $D$ and from $D$ to $A$ (see Figure~\ref{fig:extremal}). Since every path that joins two vertices in $B$ has to pass through $D$, it follows that every cycle contains at least as many vertices from $D$ as it contains from $B$. As $|B| > |D|$, this means that one cannot find a Hamilton cycle in $G$.

    It remains to show that the sizes of $A,B,C,D$ and the tournaments can be chosen to satisfy $\deg^+(x)+\deg^-(y)=\lceil (3n-3)/4\rceil-1$ for some pairs $(x,y)$ with $xy\notin E(G)$. Table~\ref{tab:extremal} gives possible values, according to the value of $n~\mbox{mod}~4$: we choose the tournaments inside $A$ and $C$ to be as regular as possible, i.e. the indegree and outdegree of every vertex differ by at most $1$. For the bipartite tournament between $B$ and $D$, we choose it so that $|N^+(b) \cap D| =\lceil |A|/2\rceil,  |N^-(b) \cap D| =|D|-\lceil |A|/2\rceil$ for every $b \in B$.

Next we check that there is a pair $x,y\in V(G)$ with $xy\notin E(G)$ and $\deg^+(x)+\deg^-(y)=\lceil (3n-3)/4\rceil-1$, as required.
Denote $a\coloneqq |A|$. Without loss of generality, we may assume that $0\le a\le |C|$ since otherwise we can simply reverse the direction of all edges of $G$ if $a\ge|C|$.
When $n=4k+1$, since $C$ spans a tournament such that the in- and outdegree of every vertex differ by at most $1$ inside $C$, there exists a vertex $u\in C$ such that $\deg^+_C(x)=\lfloor (|C|-1)/2\rfloor$. 
Then for any vertex $y\in B$, we have 
$\deg^+(x)+\deg^-(y)=(\lfloor (|C|-1)/2\rfloor+|D|)+(|A|+|D|-\lceil |A|/2\rceil)=\lceil(3n-3)/4\rceil-1$.
When $n\in\{4k, 4k+2,4k+3\}$, we observe that $\deg^+(x)+\deg^-(y)=|A|+|C|+|D|=\lceil(3n-3)/4\rceil-1$ for any two distinct vertices $x,y\in B$.
\end{proof}

\begin{remark}
One may add any number of edges that either go from $A$ to $C$ or lie within $D$
without creating a Hamilton cycle, while the resulting graph still satisfies the condition of Proposition~\ref{prop:ex}.

\end{remark}

\begin{table}[htbp]
\setlength{\tabcolsep}{14pt}  
\renewcommand{\arraystretch}{1.2}  
    \centering
    \begin{tabular}{c c c c c c }
    \hline
    $n$    & $\lceil (3n-3)/4\rceil-1$ & $|A|$ & $|B|$ & $|C|$ & $|D|$ \\\hline
    $4k$   & $3k-1$ &  $a$ & $k+1$ & $2k-1-a$  & $k$       \\
    $4k+1$ & $3k-1$ &  $a$ & $k+1$ & $2k-a$    & $k$       \\
    $4k+2$ & $3k$   &  $a$ & $k+2$ & $2k-1-a$  & $k+1$     \\
    $4k+3$ & $3k+1$ &  $a$ & $k+2$ & $2k-a$    & $k+1$     \\\hline
        
    \end{tabular}
    \caption{Parameters for the extremal example when $0\le |A|\le |C|$.}
    \label{tab:extremal}
\end{table}

\begin{definition}\label{ex-defi}
    An oriented graph $G$ of $n$ vertices is \emph{$\eta$-extremal} if there are $3$ or $4$ disjoint vertex subsets $A,B,C,D \subset V(G)$ (here $A$ or $C$ can be empty) satisfying that:
     \begin{itemize}
        \item $|A| + |C| = (1/2 \pm O(\eta)) n$,
        \item $|B| = (1/4 \pm O(\eta)) n$,
        \item $|D| = (1/4 \pm O(\eta)) n$;
    \end{itemize}
    and the edges satisfying that:

    \settasks{label = \textbullet} 
\begin{tasks}(2)

\task $e(A,B) > |A||B| -O(\eta)n^2$,
\task $e(B,C) > |B||C| - O(\eta)n^2$,
\task $e(C,D) > |C||D| - O(\eta)n^2$,
\task $e(D,A) > |A||D| - O(\eta)n^2$,
\task $e(B,D) > |A|n/8 - O(\eta)n^2$,
\task $e(D,B) > |C|n/8 - O(\eta)n^2$;
\task  $e(A) > |A|^2/2 - O(\eta)n^2$,
\task $e(C) > |C|^2/2 - O(\eta)n^2$,
\task $e(A,C) < O(\eta)n^2$,
\task $e(D) < O(\eta)n^2$.
\end{tasks}
\end{definition}

\section{Main tools}\label{s3}
In this section, we introduce some tools that will be used throughout the proof, including the Diregularity lemma, probabilistic tools, and some basic properties under our Ore-degree condition. 

\subsection{The Diregularity lemma}
We start with some definitions.
The density of bipartite graph $G[A,B]$ with vertex classes $A$
and $B$ is defined to be
\[d_G(A,B) \coloneqq \frac{e_G(A,B)}{|A||B|}.\]
We often write $d(A,B)$ if this is unambiguous. Given $\epsilon > 0$, we say that $G$ is \emph{$\epsilon$-regular} if for all subset $X\subseteq A$ and $Y\subseteq B$ with $|X| > \epsilon |A|$ and $|Y| > \epsilon |B|$ we have that $|d(X,Y) - d(A,B)| < \epsilon$. Given $d \in [0,1]$ we say that $G$ is \emph{$(\epsilon,d)$-super-regular} if it is $\epsilon$-regular and furthermore $\deg_G(a) \ge (d-\epsilon) |B|$ for all $a\in A$ and $\deg_G(b) \ge (d-\epsilon) |A|$ for all $b\in B$. (This is a slight variation of the
standard definition of $(\epsilon,d)$-super-regularity, where one requires $\deg_G(a)\ge d|B|$ and $\deg_G(b)\ge d|A|$.) 

The Diregularity lemma is a version of the Regularity lemma for digraphs due to Alon and Shapira~\cite{AS2004} (with a similar proof to the undirected version). We will use the following degree form of the Diregularity lemma, which can be derived from the standard version, in exactly the same manner as the undirected degree form (see, for example,~\cite{KO2009} for a sketch proof).

\begin{lemma}[Degree form of the Diregularity lemma]\label{lem:reg}
For every $\epsilon\in(0,1)$ there are integers $M$ and $n_0$ such that if $G$ is a digraph on $n\ge n_0$ vertices and $d\in[0,1]$ is any real number, then there is a partition of the vertices of $G$ into $V_0,V_1,\ldots,V_k$ and a spanning subdigraph $G'$ of $G$ such that the following holds:
\begin{enumerate}[label={\rm(\roman*)}]
\item $M'\le k \le M$,
\item  $|V_0|\le \epsilon n$,
\item  $|V_1| =\cdots =|V_k| \eqqcolon m$,
\item  $\deg^{+}_{G'}(x) > \deg^{+}_{G}(x) -(d + \epsilon)n$ for all vertices $x\in G$,
\item $\deg^{-}_{G'}(x) > \deg^{-}_{G}(x) - (d + \epsilon)n$ for all vertices $x\in G$,
\item for all $1 \le i \le k$ the digraph $G'[V_i]$ is empty,
\item for all $1\le i, j \le k$ and $i\neq j$ the bipartite graph $G'[V_i, V_j]$ whose vertex classes are $V_i$ and $V_j$ and whose edges are all the edges in $G'$ directed from $V_i$ to $V_j$ is $\epsilon$-regular and has density either $0$ or at least $d$.
\end{enumerate}
\end{lemma}

The vertex sets $V_1, \ldots, V_k$ are called \emph{clusters}, $V_0$ is called the \emph{exceptional set} and the vertices in $V_0$ are called \emph{exceptional vertices}. 
The last condition of the lemma says that all pairs of clusters are $\epsilon$-regular in both directions (but possibly with different densities). 
We call the spanning digraph $G'\subseteq G$ given by the Diregularity lemma the \emph{pure digraph}. 
Given clusters $V_1, \ldots, V_k$ and a digraph $G'$, the reduced digraph $R'$ with parameters $(\epsilon,d)$ is the digraph whose vertex set is $[k]$ and in which $ij$ is an edge if and only if the bipartite graph whose vertex
classes are $V_i$ and $V_j$ and whose edges are all the $V_i$--$V_j$ edges in $G'$ is $\epsilon$-regular and has density at least $d$. 
(Therefore if $G'$ is the pure digraph, then $ij$ is an edge in $R'$ if and only if there is a $V_i$--$V_j$ edge in $G'$.)

Note that the latter holds if
and only if $G'[V_i, V_j]$ is $\epsilon$-regular and has density at least $d$. It turns out that $R'$ inherits many
properties of $G$, a fact that is crucial in our proof. However, $R'$ is not necessarily oriented even
if the original digraph $G$ is, but the next lemma guarantees that we can find a reduced oriented graph $R\subseteq R'$ still approximately
satisfying the original Ore-degree condition.

\begin{lemma}[\cite{keevash2009exact,kelly2008dirac}]\label{lem:reduceR}
For every $\epsilon\in(0,1)$, there exist numbers $M'= M(\epsilon)$ and $n_0=n_0(\epsilon)$ such that the following holds. 
Let $d\in[0,1]$ with $\epsilon\le d/2$, let $G$ be an oriented graph of order $n\ge n_0$ and
let $R'$ be the reduced digraph with parameters $(\epsilon,d)$ obtained by applying the Diregularity
lemma to $G$ with $M'$ as the lower bound on the number of clusters. Then $R'$ has a spanning
oriented subgraph $R$ such that
\begin{enumerate}[label=\nlabel]
    \item for all disjoint sets $S,T\subseteq V(R)$ with $e_G(S^*,T^*)\ge 3dn^2$ we have $e_R(S,T)>d|R|^2$, where $S^*\coloneqq \cup_{i\in S}V_i$ and $T^*\coloneqq \cup_{i\in T}V_i$;
    \item for every set $S\subseteq V(R)$ with $e_G(S^*)\ge 3dn^2$ we have $e_R(S)>d|R|^2$, where $S^*\coloneqq \cup_{i\in S}V_i$;
    \item if $d \le 1-2\epsilon$ and $c\ge 0$ is such that $\deg^+(x)+\deg^-(y) \ge c|G|$ whenever $xy\notin E(G)$, then $\deg^+_R(V_i) + \deg^-_R(V_j)\ge (c -6\epsilon -2d)|R|$ whenever $V_iV_j\notin E(R)$.
\end{enumerate}
\end{lemma}

The oriented graph $R$ given by Lemma~\ref{lem:reduceR} is called the \emph{reduced oriented graph}. The spanning oriented subgraph $G^*$ of the pure digraph $G'$ obtained by deleting all the $V_i$--$V_j$ edges whenever $V_iV_j\in E(R')\setminus E(R)$ is called the \emph{pure oriented graph}. Given an oriented subgraph $S \setminus R$, the oriented subgraph of $G^*$ corresponding to $S$ is the oriented subgraph obtained from $G^*$ by deleting all those vertices that lie in clusters not belonging to $S$ as well as deleting all the $V_i$--$V_j$ edges for all pairs $V_i, V_j$ with $V_iV_j\notin E(S)$.

The following Blow-up Lemma of Koml\'{o}s, S\'{a}rk\"{o}zy, and Szemer\'{e}di~\cite{KSS1997} will help us find a set of few paths that covers almost all vertices.

\begin{lemma}[Blow-up Lemma, \cite{KSS1997}]\label{lem:blowup}
    Given a graph $F$ on $[k]$ and positive numbers $d,\Delta$, there is a positive real $\sigma_0 = \sigma_0(d,\Delta,k)$ such that the following holds for all positive numbers $\ell_1,\ldots,\ell_k$ and all $0 < \sigma \le \sigma_0$. Let $F'$ be the graph obtained from $F$ by replacing each vertex $i\in F$ with a set $V_i$ of $\ell_i$ new vertices and joining all vertices in $V_i$ to all vertices in $V_j$ whenever $ij$ is an edge of $F$. Let $G'$ be a spanning subgraph of $F'$ such that for every $ij\in E(F)$ the graph $G'[V_i,V_j]$ is $(\sigma,d)$-super-regular. Then $G'$ contains a copy of every subgraph $H$ of $F'$ with $\Delta(H) \le \Delta$. Moreover, this copy of $H$ in $G'$ maps the vertices of $H$ to the same sets $V_i$ as the copy of $H$ in $F'$, i.e. if $h\in V(H)$ is mapped to $V_i$ by the copy of $H$ in $F'$, then it is also mapped to $V_i$ by the copy of $H$ in $G'$.
\end{lemma}

\begin{proposition}[\cite{kelly2008dirac}]\label{prop:superregular}
    Let $M',n_0,D$ be positive numbers and let $\epsilon,d$ be positive reals such that $1/n_0 \ll 1/M' \ll \epsilon \ll d \ll 1/D$. Let $G$ be an oriented graph of order at least $n_0$. Let $R$ be the reduced oriented graph with parameters $(\epsilon,d)$ and let $G^*$ be the pure oriented graph obtained by successively applying first the Diregularity Lemma with $\epsilon$, $d$ and $M'$ to $G$ and then \cref{lem:reduceR}. Let $S$ be an oriented subgraph of $R$ with $\Delta(S) \le D$. Let $G'$ be the underlying graph of $G^*$. Then one can delete $2D\epsilon|V_i|$ vertices from each cluster $V_i$ to obtain subclusters $V'_i \subseteq V_i$ in such a way that $G'$ contains a subgraph $G'_S$ whose vertex set is the union of $V'_i$ and such that
    \begin{itemize}
        \setlength{\itemsep}{0pt}
        \setlength{\parsep}{0pt}
        \setlength{\parskip}{0pt}
        \item $G'_S[V'_i,V'_j]$ is $(\sqrt{\epsilon},d-4D\epsilon)$-super-regular whenever $ij\in E(S)$,
        \item $G'_S[V'_i,V'_j]$ is $\sqrt{\epsilon}$-regular and has density $d-4D\epsilon$ whenever $ij\in E(R)$.
    \end{itemize}
\end{proposition}



\subsection{Probabilistic tools}
We will rely on the following form of Chernoff's bounds that in particular provides large concentration inequalities for binomial random variables.

\begin{lemma}
    [Chernoff's bounds]\label{lem:chernoff}
    Let $\bX_1, \ldots, \bX_n$ be i.i.d. $(0,1)$-valued random variables, and let $\bS_n \coloneqq \sum_{i=1}^n \bX_i$. Let us write $\mu\coloneqq \esp{\bS_n}$. Then
    \begin{enumerate}[label={\rm(\roman*)}]
        \item for every $0 <\delta < 1$,
         \[ \pr{\bS_n \le (1-\delta)\mu } \le e^{-\delta^2\mu/2};\]
         \label{it:upper}
        \item for every $\delta > 0$, 
        \[ \pr{\bS_n \ge (1+\delta)\mu } \le e^{-\delta^2\mu/(2+\delta)}.\]
        \label{it:lower}
    \end{enumerate}
\end{lemma}

\begin{lemma}\label{lem:absorberfamily}
    Let $\sigma$ be a real number with $0 < \sigma < 1$ and $t$ be an integer with $1 \le t \le 4$. Then there exists an integer $n_0$ such that whenever $n\ge n_0$ the following holds. Let $G$ be an oriented graph on $n$ vertices and $U$ be a vertex subset with $|U| = n' \ge n/2$. Let $S\subseteq V(G) \times V(G)$ be a set of pairs of vertices (not necessary to be distinct). Suppose that $\cL(u,v)$ is a family of ordered $t$-tuples of $U$ such that $|\cL(u,v)| \ge \sigma n^t$ for every $(u,v)\in S$. Then there exists a family $\cF \subseteq \bigcup \cL(u,v)$ of vertex-disjoint $t$-tuples, which satisfies the following properties:
    \[|\cF| \le 2^{-6}\sigma n,~~~~~~|\cL(u,v) \cap \cF| \ge 2^{-10}\sigma^2 n\]
    for all $(u,v)\in S$.
\end{lemma}

\begin{proof}
    Choose $n_0$ large so that
    \begin{align}
        \exp(-\sigma n_0/(3\times 2^7)) + n_0^2 \exp(-\sigma^2 n_0/2^8) \le 1/6\label{eq:4.13}
    \end{align}
    Let us choose a family $\cF'$ of $t$-tuples in $U$ by selecting each of the $n'!/(n'-t)!$ possible tuples independently at random with probability
    \[p = 2^{-7} \sigma \frac{(n'-t)!}{(n'-1)!} \ge 2^{-7} \sigma n'^{(-t+1)}.\]
    Notice that
    \[\bE[|\cF'|] = p \frac{n'!}{(n'-t)!} = 2^{-7} \sigma n',\]
    \[\bE[|\cL(u,v) \cap \cF'|] = p |\cL(u,v)| \ge 2^{-7} \sigma^2 n'\]
    for every $(u,v)\in S$. Then by Chernoff's bound, the union bound and (\ref{eq:4.13}), with probability $2/3$, the families $\cF'$ satisfy the following properties:
    \[|\cF'| \le 2\bE[|\cF'|] = 2^{-6} \sigma n' \le 2^{-6} \sigma n,\]
    \[|\cL(u,v) \cap \cF'| \ge 2^{-1} \bE[|\cL(u,v) \cap \cF'| \ge 2^{-8} \sigma^2 n'.\]
    We say two tuples are \emph{intersecting} if they share at least one common vertex. We can bound the expected number of pairs of tuples in $\cF'$ that are intersecting from above be
    \[\frac{n'!}{(n'-t)!} \cdot t^2 \cdot \frac{(n'-1)!}{(n'-t)!} \cdot p^2 \le 2^{-10} \sigma^2 n'.\]
    Thus, using Markov's inequality, it holds that with probability at least $1/2$,
    \[\cF' \text{ contains at most } 2^{-9} \sigma^2 n' \text{ intersecting pairs of tuples}.\]
    Hence, with positive probability the family $\cF'$ satisfies properties. Remove one tuple in each intersecting pair in $\cF'$. We get a family $\cF$ consisting of pairwise disjoint tuples, which satisfies
    \[|\cL(u,v) \cap \cF| \ge 2^{-8} \sigma^2 n' - 2^{-9} \sigma^2 n' \ge 2^{-10} \sigma^2 n.\]
\end{proof}

\subsection{Basic properties}
We can obtain the following properties by Ore-degree condition.

\begin{proposition}\label{le:delta0}
Let $G$ be an $n$-vertex oriented graph with Ore-degree condition. We have that $\delta^0(G)\ge n/8$.
\end{proposition}

\begin{proof}
Suppose that there is a vertex $x$ with $\deg^+(x)< n/8$. Let $W\coloneqq\{y:xy\notin E(G)\}$. Clearly we have $|W|\ge 7n/8$. For each vertex $y\in W$, $\deg^-(y)\ge (3n-3)/4-\deg^+(x)\ge (5n-3)/8$ according to Ore-degree condition.
Since $G$ is an oriented graph, we have
\begin{align}
    \binom{n}{2}\ge\sum_{\substack{y\in W}}\deg^-(y)\ge (5n-3)/8|W|> 17n^2/32\nonumber,
\end{align}
which is a contradiction.
By symmetry, we can also prove $\delta^-(G)\ge n/8$. Therefore, $\delta^0(G)\ge n/8$.
\end{proof}

\begin{lemma}
    \label{lm:emptyset}
    Assume that $1/n \ll \sigma \ll 1$. Let $G$ be an $n$-vertex oriented graph with Ore-degree condition. If a vertex subset $X\subseteq V(G)$ satisfies that $e(X) \le \sigma n^2$, then $|X| \le \frac{n}{4} + 21\sigma n$.
\end{lemma}

\begin{proof}
    We may assume that $|X| \ge n/4$. Let us consider all pairs of vertices $u,v\in X$ with $uv\notin E(G)$. By Ore-degree condition we have
    \begin{align*}
        \frac{3n-3}{4} \left( |X|(|X|-1) - e(X) \right) &\le \sum_{u,v\in X, uv\notin E(G)} (\deg^+(u) + \deg^-(v))\\
        &\le |X| \sum_{v\in X} (\deg^+(v) + \deg^-(v))\\
        &\le |X| (|X| (n-|X|) + 2e(X)).
    \end{align*}
    Dividing both sides by $|X|^2$, we obtain
    \[\frac{3n-3}{4} \left( \frac{|X|-1}{|X|} - \frac{e(X)}{|X|^2} \right) \le n-|X| + \frac{2e(X)}{|X|}.\]
    Hence we have
    \[|X| \le n + 8\sigma n - \frac{3n-3}{4} \cdot (1 - 16\sigma - \frac{1}{|X|}) \le \frac{n}{4} + 21\sigma n.\]
\end{proof}

\section{Non-extremal case: proof of Theorem~\ref{thm:non-extremal}}\label{proof-non ex}
The proof of Theorem~\ref{thm:non-extremal} relies on the absorbing method, a powerful combinatorial technique initiated by R\"{o}dl, Ruci\'{n}ski, and Szemer\'{e}di~\cite{RRS2006}. In the present case, we may reduce \cref{thm:non-extremal} to Lemmas \ref{lem:connectinglemma}-\ref{lem:almostcover} formulated later in this section.

To begin with, we state the following Connecting Lemma. One may observe that Ore-degree condition indeed implies that $G$ is dense with average degree almost $3n/4$. Hence it is expected that two vertices can be linked by sufficiently many short paths.

\begin{lemma}[Connecting Lemma]\label{lem:connectinglemma}
    Let $G$ be an $n$-vertex oriented graph satisfying Ore-degree condition. For every pair of distinct vertices $(u,v)$, if $uv\notin E(G)$, then there is an integer $1\le k\le 3$ such that $G$ contains at least $2^{-24}n^k$ of $(u,v)$-path of length $k+1$.
\end{lemma}

\begin{lemma}[Reservoir Lemma]\label{lem:reservoirlemma}
    Suppose $0 < 1/n \ll \alpha_3 \ll \alpha_2 \ll 1$. Let $G$ be an $n$-vertex oriented graph satisfying Ore-degree condition. For any vertex subset $X\subseteq V(G)$ with $|X| \le 2^{-2}\alpha_2 n$ there exists a set of vertices $R \subseteq V(G)\setminus X$ of size at most $2^{-3}\alpha_3 n$ having the following property:

    For every $S\subseteq R$ with $|S| \le 2^{-11}\alpha_3^2 n$ and for every pair of vertices $(x,y)$ with $x,y\in V(G) \setminus R$ and $xy\notin E(G)$, there exists an $(x,y)$-path with at most $3$ internal vertices from $R\setminus S$.
\end{lemma}

\begin{lemma}[Absorbing Lemma]\label{lem:absorbinglemma}
    Suppose $0 < 1/n \ll \alpha_2 \ll \eta \ll 1$. Let $G$ be an $n$-vertex oriented graph satisfying Ore-degree condition. If $G$ is not $\eta$-extremal, then there exists a directed path $P_{abs}$ with at most $2^{-2}\alpha_2 n$ vertices such that, for every $U\subseteq V(G)\setminus V(P_{abs})$ with $|U| \le 2^{-10}\alpha_2^2 n$, $G[V(P_{abs}) \cup U]$ contains a spanning directed path having the same end-points as $P_{abs}$.
\end{lemma}

\begin{lemma}[Covering Lemma]\label{lem:almostcover}
    Suppose $1/n \ll 1/M \ll \epsilon \ll d \ll \eta\ll 1$. 
    Let $G$ be an $n$-vertex oriented graph satisfying Ore-degree condition, and let $X\subseteq V(G)$ with $|X|\le dn$. Then at least one of the following holds:
    \begin{enumerate}[label={\rm(\roman*)}]
        \item there exists a collection of at most $M/2$ vertex-disjoint paths that covering all but at most $5\epsilon n$ vertices of $V(G) \setminus X$;
        \item $G$ is $\eta$-extremal.
    \end{enumerate}
\end{lemma}

We conclude this section by proving our main results for the non-extremal case, assuming the four lemmas stated above. 
While the proof below does not directly invoke \cref{lem:connectinglemma}, this lemma will play an essential role in establishing both \cref{lem:absorbinglemma} and \cref{lem:reservoirlemma} in the subsequent section.

\begin{proof}[Proof of \cref{thm:non-extremal}]
    Let $M \in \mathbb{N}$ and introduce additional constants such that
    \[1/n \ll 1/M \ll \epsilon \ll \alpha_3 \ll \alpha_2 \ll \alpha_1 \ll d \ll \eta \ll 1.\]
    Let us assume that $G$ is not $\eta$-extremal. We first apply \cref{lem:absorbinglemma} with parameter $\alpha_2,\eta$ to get an absorbing path $P_{abs}$ with $|V(P_{abs})| \le 2^{-2}\alpha_2 n$ such that, for every $U\subseteq V(G)\setminus V(P_{abs})$ with $|U| \le 2^{-10}\alpha_2^2 n$, $G[V(P_{abs}) \cup U]$ contains a spanning directed path having the same end-points as $P_{abs}$. Next we appeal with $\alpha_3,\alpha_2$ to \cref{lem:reservoirlemma}, thus getting a reservoir set $R\subseteq V(G)\setminus V(P_{abs})$ with $|R| \le 2^{-3}\alpha_3 n$. Recall that we assume that $G$ is not $\eta$-extremal. Applying \cref{lem:almostcover} with $X \coloneqq V(P_{abs})\cup R$ we can get a collection of vertex-disjoint paths, say, $\cP \coloneqq \{P_1,P_2,\ldots,P_k\}$ with $k\le M/2$, and a vertex subset $Y \coloneqq V(G) \setminus (X \cup V(\cP))$ with $|Y| \le 5\epsilon n$. For convenience let us write $P_0 \coloneqq P_{abs}$ and for every $0\le i\le k$, we say $P_i$ is a $(u_i,v_i)$-path.

    We then show how to find a Hamilton cycle. To begin with, for the paths $P_0$ and $P_1$, either we have $v_0u_1\in E(G)$, or by \cref{lem:reservoirlemma} there exists an $(v_0,u_1)$-path with at most $3$ internal vertices from $R$. We claim that one may greedily apply \cref{lem:reservoirlemma} to find a $(v_i,u_{i+1})$-path for every $0\le i\le k$ (write $u_0=u_{k+1}$), since the total number of used vertices in the reservoir set $R$ is at most $3k \le 2^{-11}\alpha_3^2 n$ by our setting of constants $1/n \ll 1/M \ll \alpha_3$.

    So far we have gotten a cycle $C_0$ that contains $P_{abs}$, $\cP$, and a fraction of vertices in $R$. Let $U\coloneqq V(G) \setminus V(C_0)$. Since $|U| \le 2^{-3}\alpha_3 n + 5\epsilon n \le 2^{-10} \alpha_2^2 n$, applying \cref{lem:absorbinglemma} there exists a directed path $P$ spanning $V(P_{abs}) \cup U$ having the same end-points as $P_{abs}$. Therefore, we get a Hamilton cycle if we replace $P_{abs}$ by $P$ in $C_0$, as desired.
\end{proof}

\section{Absorption Method}\label{sec:abs}
In this section we devote to prove Lemmas \ref{lem:connectinglemma}-\ref{lem:absorbinglemma}, the key absorbing lemmas of the proof for non-extremal case.

\subsection{Proof of Connecting Lemma}
We start with the following definition. 

\begin{definition}
    For a given integer $1\le k\le 3$, a \emph{$k$-connector} of a pair of distinct vertices $(u,v)$ is a $(w_1,w_2)$-path $P$ containing exactly $k$ vertices such that $uw_1,w_2v\in E(G)$. Note that a $1$-connector refers to a vertex $w$ such that $uw,wv\in E(G)$.
\end{definition}

\begin{proof}[Proof of \cref{lem:connectinglemma}]
    Assume for the sake of contradiction that there are two vertices $u,v\in V(G)$ with $uv\notin E(G)$, such that for every integer $1\le k\le 3$, $G$ contains less than $2^{-24}n^k$ of $(u,v)$-paths of length $k+1$. For convenience, let $d^+ \coloneqq \deg^+(u)$ and $d^- \coloneqq \deg^-(v)$. Since $uv\notin E(G)$, Ore-degree condition implies that $d^+ + d^- \ge (3n-3)/4$. 
    We first analyze the neighborhoods of $u$ and $v$.
    By our assumption, the number of $1$-connectors is less than $2^{-24}n$, so $|N^+(u) \cap N^-(v)| \le 2^{-24}n$. 
    Similarly, since there are less than $2^{-24}n^2$ $2$-connectors, we have $e(N^+(u),N^-(v)) \le 2^{-24}n^2$.

    Now, define
    \begin{align*}
        \overline{N}^{++}(u) &\coloneqq \{w\notin N^+(u): |N^+(u) \cap N^-(w)| \ge 2^{-8}n\},\\
        \overline{N}^{--}(v) &\coloneqq \{w\notin N^-(v): |N^-(v) \cap N^+(w)| \ge 2^{-8}n\}.
    \end{align*}
    
    Let us consider all pairs of vertices $(x,y)$ with $x\in N^+(u)$ and $y\in N^-(v)$ such that $xy\notin E(G)$. By Ore-degree condition, we have
    \[\frac{3n-3}{4} \left(d^+ \cdot (d^--1) - 2^{-24}n^2\right) \le \sum_{
    \substack{x\in N^+(u),y\in N^-(v)\\xy\notin E(G)}} \deg^+(x) + \deg^-(y).\]
    For the right side, we may estimate that
    \begin{align*}
        \sum_{x\in N^+(u)} \deg^+(x) &\le e(G[N^+(u)]) + e(N^+(u),\overline{N}^{++}(u)) + 2^{-8}n^2\\
        &\le \binom{d^+}{2} + d^+ \cdot |\overline{N}^{++}(u)| + 2^{-8}n^2.
    \end{align*}
    A similar inequality holds for $N^-(v)$ as
    \[\sum_{y\in N^-(v)} \deg^-(y) \le \binom{d^-}{2} + d^- \cdot |\overline{N}^{--}(v)| + 2^{-8}n^2.\]
    Combining above three inequalities we can get
    \begin{align*}
        \frac{3n}{4} \left(d^+ \cdot d^- - 2^{-23}n^2\right) \le &d^+ \cdot \binom{d^-}{2} + d^- \cdot \binom{d^+}{2}
        + d^+ \cdot d^- \cdot (|\overline{N}^{++}(u)| + |\overline{N}^{--}(v)|)\\
        &+ (d^+ + d^-) \cdot 2^{-8}n^2.
    \end{align*}
    Dividing both sides by $d^+ \cdot d^-$ we have
    \[\frac{3n}{4} \left(1 - \frac{2^{-23}n^2}{d^+ \cdot d^-}\right) \le \frac{d^+ + d^-}{2} + (|\overline{N}^{++}(u)| + |\overline{N}^{--}(v)|) + \frac{d^+ + d^-}{d^+ \cdot d^-} 2^{-8}n^2.\]
    \cref{le:delta0}, the semi-degree condition, implies that $d^+ \cdot d^- \ge d^+ \cdot (\frac{3n-3}{4} - d^+) > 2^{-4}n^2$. Thus,
    \[|\overline{N}^{++}(u)| + |\overline{N}^{--}(v)| > \frac{11n}{16} - \frac{9(d^+ + d^-)}{16}.\]
    Additionally, we have $|\overline{N}^{++}(u) \cap N^-(v)| \le 2^{-16}n$; otherwise we have $e(N^+(u),N^-(v)) > 2^{-24}n^2$, which contradicts our assumption. Similarly, $|\overline{N}^{--}(v) \cap N^+(u)| \le 2^{-16}n$. Furthermore, the number of $3$-connectors implies that $|\overline{N}^{++}(u) \cap \overline{N}^{--}(v)| \le 2^{-8}n$. Hence, we have
    \[n \ge |N^+(u)| + |N^-(v)| + |\overline{N}^{++}(u)| + |\overline{N}^{--}(v)| - 2^{-24}n - 2^{-15}n - 2^{-8}n > \frac{65n}{64} - 2^{-7}n > n,\]
    where we use that $|N^+(u) \cap N^-(v)| \le 2^{-24}n$, $|\overline{N}^{++}(u) \cap N^-(v)| + |\overline{N}^{--}(v) \cap N^+(u)| \le 2^{-15} n$, and $|\overline{N}^{++}(u) \cap \overline{N}^{--}(v)| \le 2^{-8}n$. This yields a contradiction.
\end{proof}

\subsection{Proof of Reservoir Lemma}
We then prove~\cref{lem:reservoirlemma} by using Lemmas~\ref{lem:absorberfamily},~\ref{lem:connectinglemma} directly.

\begin{proof}[Proof of \cref{lem:reservoirlemma}]
    Suppose $0 < 1/n \ll \alpha_3 \ll \alpha_2 \ll 1$, and let $G$ be an oriented graph satisfying Ore-degree condition. By \cref{lem:connectinglemma}, for every pair of vertices $(u,v)$ with $uv\notin E(G)$, there exists an integer $1\le k\le 3$ such that $(u,v)$ admits at least $2^{-24}n^k$ of $k$-connectors. Let $S_1$ be the set of pairs of vertices in $G$ having at least $2^{-24}n$ of $1$-connectors. 
    Fix a subset $X\subseteq V(G)$ with $|X|\le 2^{-2}\alpha_2n$. For every $(u,v)\in S_1$, let $\cL_1(u,v)$ be the set of $1$-connectors of $(u,v)$ which do not intersect with $X$. Then
    \[|\cL_1(u,v)| \ge 2^{-24}n - |X| \ge \alpha_3 n.\]
    Applying \cref{lem:absorberfamily} with $\sigma = \alpha_3$ and $U = V(G) \setminus X$, we obtain a family $\cF_1$ of $1$-connectors such that
    \begin{align}
        \label{connectorcondition}
        |\cF_1| \le 2^{-6} \alpha_3 n,~~~~|\cL_1(u,v) \cap \cF_1| \ge 2^{-10} \alpha_3^2 n \text{ for all } (u,v)\in S_1.
    \end{align}

    Set $X_1 \coloneqq X \cup V(\cF_1)$. Then $|X_1|\le 2^{-1}\alpha_2 n$ by the choice of $\alpha_3 \ll \alpha_2$. By \cref{lem:absorberfamily}, one can find a family $\cF_2$ consisting of vertex-disjoint $2$-connectors from $V(G)\setminus X_1$. By the same argument, we have $\cF_3$ consisting of vertex-disjoint $3$-connectors from $V(G)\setminus (X_1\cup V(\cF_2))$. Note that $\cF_2$ and $\cF_3$ both satisfy the property analogous to (\ref{connectorcondition}) with respect to $\cL_2,S_2$ and $\cL_3,S_3$, respectively.

    Let $R \coloneqq \bigcup_i V(\cF_i)$. It remains to show that $R$ satisfies the desired property. In fact, for any $S\subseteq R$ with $|S| \le 2^{-11} \alpha_3^2 n$ and a pair of vertices $(x,y)$ with $xy\notin E(G)$, for some $1\le k\le 3$ we have $|\cL_k(x,y) \cap \cF_k| \ge 2^{-10}\alpha_3^2 n$, which follows from \cref{lem:connectinglemma}. Thus, we can find a $k$-connector in $\cL_k(x,y)$ that does not intersect with $S$, and thus get an $(x,y)$-path with at most $3$ internal vertices from $R\setminus S$.
\end{proof}

\subsection{Proof of Absorbing Lemma}\label{subsec:5.3}
The present section is dedicated to the proof of \cref{lem:absorbinglemma}, in which we state that the absorber we shall use will simply be a path. A simple way to absorb a vertex or a directed path is to insert them into two consecutive vertices. Hence we have the following definition.

\begin{definition}
    Let $G$ be an $n$-vertex oriented graph. A \emph{strong absorber} of a pair of vertices $(u,v)$ (not necessary to be distinct) is a set of two vertices $w,z\in V(G)$ such that $wz,wu,vz\in E(G)$.
    We say that $(u,v)$ is \emph{$\alpha_1$-strongly absorbable} if it has at least $\alpha_1 n^2$ strong absorbers. Specially, we refer to these definitions to a single vertex when $u,v$ are identified.
\end{definition}

However, there may exist some vertices that are not strongly absorbable even if $G$ is not $\eta$-extremal. Let $v$, for instance, be such a vertex. Our strategy is to insert $v$ into two non-consecutive vertices on the absorbing path, while the directed path between these two vertices is removed. If this path can be strongly absorbed, then we succeed in absorbing the vertex $v$. So we introduce the following definition.

\begin{definition}
    Let $G$ be an $n$-vertex oriented graph. An \emph{$\alpha_1$-weak absorber} of a pair of vertices $(u,v)$ (not necessary to be distinct) is a set of four vertices $w,w',z',z\in E(G)$ such that $ww',wu,z'z,vz\in E(G)$ and $(w',z')$ is $\alpha_1$-strongly absorbable.
    We say that $(u,v)$ is \emph{$(\alpha_2,\alpha_1)$-weakly absorbable} if it has at least $\alpha_2 n^4$ of $\alpha_1$-weak absorbers. Specially, we refer to these definitions to a single vertex when $u,v$ are identified.
\end{definition}

The following proposition declares that there always exist sufficiently many strongly absorbable pairs in any dense subgraph.

\begin{proposition}\label{cl:densestrong}
    Let $1/n \ll \alpha \ll \beta < 1$. Let $G$ be an $n$-vertex oriented graph, and let $X,Y$ be two vertex subsets of $V(G)$ (not necessary to be disjoint). If $e(Y,X) = \beta n^2$, then there are at least $(\beta^4/32 - \alpha) n^2$ pairs of vertices $(x,y) \in (X,Y)$ that are $\alpha$-strongly absorbable.
\end{proposition}

\begin{proof}
    Consider the anti-directed paths of length $3$ in the bipartite graph $G[Y,X]$, that is, $xuvy$ with $x,v \in X$, $u,y \in Y$ and $ux,uv,yv \in E(G)$. Let $k$ be the number of such paths. Since $G[Y,X]$ has average degree at least $\beta n$, there exist $Y'\subseteq Y$ and $X'\subseteq X$ such that $G[Y',X']$ has minimum degree at least $\beta n/2$. Hence we have
    \[k \ge \frac{\beta n}{2} \cdot \frac{\beta n}{2} \cdot \left(\frac{\beta n}{2} - 1\right)\cdot \left(\frac{\beta n}{2} - 2\right) \ge \frac{\beta^4 n^4}{32}.\]
    
    Let $m$ be the number of $\alpha$-strongly absorbable pairs of $(X,Y)$. The number of anti-directed paths whose endpoints are some strongly absorbable pair is at most $m\cdot n^2$. While the number from those that are not strongly absorbable is at most $n^2\cdot \alpha n^2$. Hence we have
    \[k \le m \cdot n^2 + \alpha n^4.\]
    Hence the number of strongly absorbable pairs is at least $(\beta^4/32 - \alpha) n^2$.
\end{proof}

The next lemma shows that if $G$ is not $\eta$-extremal, then every vertex of $G$ is either strongly or weakly absorbable.
\begin{lemma}\label{lm:absorbing}
    Let $1/n \ll \alpha_2 \ll \alpha_1 \ll  \eta\ll 1$. Let $G$ be an $n$-vertex oriented graph satisfying Ore-degree condition. If $G$ is not $\eta$-extremal, then every vertex $v\in V(G)$ is either $\alpha_1$-strongly absorbable or $(\alpha_2,\alpha_1)$-weakly absorbable.
\end{lemma}

We first prove \cref{lem:absorbinglemma} under the assumption that \cref{lm:absorbing} holds; the proof of \cref{lm:absorbing} will be provided later.

\begin{proof}[Proof of \cref{lem:absorbinglemma}]
    Suppose $0 < 1/n \ll \alpha_2 \ll \eta \ll 1$. Let $G$ be an $n$-vertex oriented graph satisfying Ore-degree condition, and let $G$ is not $\eta$-extremal. 
    Note that we may use $(u,u)$ to refer to $u$ so that we can transfer the absorbing property of one vertex to that of two vertices. Let $S_s\subseteq V(G)\times V(G)$ be the set of all pairs of vertices $(u,v)$ such that $(u,v)$ is $\alpha_1$-strongly absorbable. Let $S_w$ be the set of all pairs $(u,u)$ such that $u$ is $(\alpha_2,\alpha_1)$-weakly absorbable. Let $\cL_s(u,v)$ and $\cL_w(u,v)$ be the set of strong absorbers and weak absorbers of $(u,v)$, respectively.

    First applying \cref{lem:absorberfamily} to $S_w$ and $\cL_w$ with $\sigma= \alpha_2$ and $t=4$, we can find a family $\cF_w$ of weak absorbers such that
    \begin{align}
        \label{weakabcondition}
        |\cF_w| \le 2^{-6} \alpha_2 n,~~~~|\cL_w(u,u) \cap \cF_w| \ge 2^{-10} \alpha_2^2 n \text{ for all } (u,u)\in S_w.
    \end{align}
    Then for every $(u,v)\in S_s$, let $\cL'_s(u,v)$ be the subset of $\cL_s(u,v)$ by removing all elements that intersect with $V(\cF_w)$. Since $\alpha_2 \ll \alpha_1$ we can find a family $\cF_s$ of strong absorbers satisfying the property analogous to (\ref{weakabcondition}) with respect to $S_s$ and $\cL'_s(u,v)$. We claim that there exists a path $P_{abs}$ satisfying that:
    \begin{enumerate}[label={\rm(\roman*)}]
        \item for every $wz\in \cF_s$, $P_{abs}$ contains the edge $wz$.
        \item for every $ww'z'z\in \cF_w$, $P_{abs}$ contains the edges $ww'$ and $z'z$. Furthermore, $P_{abs}$ contains a $(w',z')$-path which does not intersect with any other absorber in $\cF_s \cup \cF_w$.
        \item $P_{abs}$ has at most $\alpha_2 n/4$ vertices.
    \end{enumerate}
    In fact, $P_{abs}$ can be found by greedily connecting all edges of these absorbers in order. \cref{lem:connectinglemma} guarantees that each connecting operation consumes at most $3$ extra vertices. Hence the number of vertices in $P_{abs}$ is at most $7|\cF_s\cup \cF_w| \le \alpha_2 n/4$.

    By \cref{lm:absorbing}, every vertex is either strongly or weakly absorbable, and hence can be absorbed by $P_{abs}$ using exactly one strong absorber and at most one weak absorber. Thus any vertex subset $U\subseteq V(G)\setminus V(P_{abs})$ with $|U| \le 2^{-10}\alpha_2^2 n$ can be greedily absorbed by $P_{abs}$.
\end{proof}

We now prove \cref{lm:absorbing}.

\begin{proof}[Proof of \cref{lm:absorbing}]
    suppose to the contrary that $G$ is not $\eta$-extremal and there is a vertex $v$ that is neither $\alpha_1$-strongly absorbable nor $(\alpha_2,\alpha_1)$-weakly absorbable. Let $A_0 \coloneqq N^+(v)$, $C_0 \coloneqq N^-(v)$, and $R \coloneqq V(G) \setminus (A_0 \cup C_0 \cup \{v\})$. Let $\gamma$ be a positive real number such that $\alpha_1 \ll \gamma \ll \eta$. We define some specific subsets of $A_0,C_0,R$ as follows.
    \begin{align*}
        &A_1 \coloneqq \{w \in A_0 \mid \deg^+_{A_0}(w) \ge \gamma n\};\\
        &C_1 \coloneqq \{w \in C_0 \mid \deg^-_{C_0}(w) \ge \gamma n\};\\
        &R_A \coloneqq \{w \in R \mid \deg^+_{A_0}(w) \ge \gamma n\};\\
        &R_C \coloneqq \{w \in R \mid \deg^-_{C_0}(w) \ge \gamma n\};\\
        &A_2 \coloneqq A_0 \setminus A_1;\\
        &C_2 \coloneqq C_0 \setminus C_1.
    \end{align*}
    Since $v$ is neither strongly nor weakly absorbable, the edges between some vertex subsets must be almost empty:
    \begin{align}
        \label{eq:4.1}
        e(C_0,A_0) \le \alpha_1 n^2;
    \end{align}
    \[e(A_1 \cup R_A,C_1 \cup R_C) \le \left(32\left(\frac{\alpha_2}{\gamma^2} + \alpha_1\right)\right)^{1/4} n^2.\]
    The first inequality immediately follows from the assumption that $v$ is not $\alpha_1$-strongly absorbable. In fact, if we suppose that the second inequality does not hold, then there exist $\alpha_2 n^2/\gamma^2$ pairs of vertices $(x,y)$, where $x\in C_1 \cup R_C $ and $y\in A_1 \cup R_A$, that are $\alpha_1$-strongly absorbable by \cref{cl:densestrong}. So the number of $(\alpha_2,\alpha_1)$-weak absorbers of $v$ is at least $(\alpha_2 n^2/\gamma^2) \cdot (\gamma n)^2 \ge \alpha_2 n^4$, which contradicts our assumption that $v$ is not $(\alpha_2,\alpha_1)$-weakly absorbable.
    We may further assume that $\alpha_1 \le \gamma^5$ and $\alpha_2 \le \gamma^7$ since $\alpha_1,\alpha_2 \ll \gamma$. Then we have
    \begin{align}
        \label{eq:4.2}
        e(A_1 \cup R_A,C_1 \cup R_C) \le \gamma n^2 /2.
    \end{align}
    Let us denote $R'_A \coloneqq R_A \setminus R_C$, $R'_C \coloneqq R_C \setminus R_A$, $R_{AC} \coloneqq R_A \cap R_C$, and $R' \coloneqq R \setminus (R_A \cup R_C)$. For each of the vertex subsets defined above, we use the corresponding lowercase letter to denote its size. For instance, we write $a_1 \coloneqq |A_1|$, $r_a \coloneqq |R_A|$, etc.

    We need the following inequalities, which will be proved at the end of this subsection.
    \begin{equation}
        \label{eq:c1a1}
        \tag{$C_1A_1$}
        \frac{3n}{4} a_1 c_1 \le c_1\binom{a_1}{2} + a_1\binom{c_1}{2} + c_1\cdot e(R,A_1) + a_1\cdot e(C_1,R) + \gamma \cdot n^3.
    \end{equation}
    \begin{equation}
        \label{eq:a1c1}
        \tag{$A_1C_1$}
        \frac{3n}{4} a_1 c_1 \le c_1\binom{a_1}{2} + a_1\binom{c_1}{2} + 2 a_1 c_1 (a_2 + c_2) + c_1\cdot e(A_1,R) + a_1\cdot e(R,C_1) + \gamma \cdot n^3.
    \end{equation}
    \begin{equation}
        \label{eq:c2a1}
        \tag{$C_2A_1$}
        \frac{3n}{4} a_1 c_2 \le c_2\binom{a_1}{2} + a_1 c_1 c_2 + c_2\cdot e(R,A_1) + a_1\cdot e(C_2,R) + \gamma \cdot n^3.
    \end{equation}
    \begin{equation}
        \label{eq:c1a2}
        \tag{$C_1A_2$}
        \frac{3n}{4} a_2 c_1 \le a_1 a_2 c_1 + a_2\binom{c_1}{2} + c_1\cdot e(R,A_2) + a_2\cdot e(C_1,R) + \gamma \cdot n^3.
    \end{equation}
    \begin{equation}
        \label{eq:c2a2}
        \tag{$C_2A_2$}
        \frac{3n}{4} a_2 c_2 \le a_1 a_2 c_2 + a_2 c_1 c_2 + c_2\cdot e(R,A_2) + a_2\cdot e(C_2,R) + \gamma \cdot n^3.
    \end{equation}
    
    \begin{equation}
        \label{eq:vra'}
        \tag{$vR'_A$}
        r'_a \left( c_1 + r - \frac{n}{2} \right) \le 2\gamma \cdot n^2.
    \end{equation}
    \begin{equation}
        \label{eq:vrc'}
        \tag{$vR'_C$}
        r'_c \left( a_1 + r - \frac{n}{2} \right) \le 2\gamma \cdot n^2.
    \end{equation}
    \begin{equation}
        \label{eq:vrac}
        \tag{$vR_{AC}$}
        r_{ac} \left( r_{ac} + r - \frac{n}{2} \right) \le 2\gamma \cdot n^2.
    \end{equation}
    
    \begin{equation}
        \label{eq:a1v}
        \tag{$A_1v$}
        \frac{3n}{4} a_1 \le a_1 c_1 + a_1 c_2 + \binom{a_1}{2} + a_1 (a_2 + c_2) + e(A_1,R) + \gamma \cdot n^2.
    \end{equation}
    \begin{equation}
        \label{eq:vc1}
        \tag{$vC_1$}
        \frac{3n}{4} c_1 \le a_1 c_1 + a_2 c_1 + \binom{c_1}{2} + c_1 (a_2 + c_2) + e(R,C_1) + \gamma \cdot n^2.
    \end{equation}
    
     We now present the following claim to characterize $e(R,A_i)$ and $e(C_i,R)$ where $0\le i \le 2$.
    \begin{cl}\label{cl:rarc}
        Let $i$ be an integer with $1\le i\le 2$. If $a_i \ge n/t$ for some positive constant $t>1$, then it holds that
        \[\frac{e(R,A_i)}{a_i} \le r_a + t \gamma n.\]
        Moreover, if $r_a \le \frac{e(R,A_i)}{a_i} + \sigma n$, then $e(R,A_i) \ge a_ir_a - \sigma n^2$.
    \end{cl}
    \begin{proof}
        The second result holds trivially if we divide both sides of that inequality by $a_i$. We then prove the first inequality. By the definition of $R_A$, we have
        \[e(R,A_i) = e(R_A,A_i) + e(R\setminus R_A,A_i) \le r_a a_i + \gamma n^2.\]
        Hence we have
        \[\frac{e(R,A_i)}{a_i} \le r_a + t \gamma n,\]
        since $a_i\ge n/t$.
    \end{proof}

    We finish the proof of this lemma by analysing the different cases according to the order of $A_1$ and $C_1$.

   \textbf{Case 1:} $a_1\le \gamma^{1/3} n$, $c_1\le \gamma^{1/3} n$.
   
    Note that $\delta^0(G)\ge n/8$ by \Cref{le:delta0}, together with (\ref{eq:c2a2}) and \cref{cl:rarc} we can get
    \begin{align}
        \label{eq:4.3}
        \frac{3n}{4} \le a_1 + c_1 + \frac{e(R,A_2)}{a_2} + \frac{e(C_2,R)}{c_2} + 100\gamma n \le a_1 + c_1 + r_a + r_c + 120\gamma n.
    \end{align}
    Notice that $r = r_a + r_c - r_{ac} + r'$,  we have $r_a + r_c \ge 3n/4 - 3\gamma^{1/3} n > n/2$ by (\ref{eq:4.3}). Together with (\ref{eq:vra'}), (\ref{eq:vrc'})~and~(\ref{eq:vrac}), we can get
    \[\frac{r'_a \cdot c_1 + r'_c \cdot a_1 + 2(r_{ac})^2}{r'_a + r'_c + 2r_{ac}} \le \frac{n}{2} - r + 16\gamma n.\]
   It also holds that
    \[\frac{r'_a \cdot c_1 + r'_c \cdot a_1 + 2(r_{ac})^2}{r'_a + r'_c + 2r_{ac}} \ge \frac{2(r_{ac})^2}{r_a + r_c} \ge \frac{2(r_a + r_c - r)^2}{r_a + r_c}.\]
    Hence 
    \[\frac{2(r_a + r_c - r)^2}{r_a + r_c} \le \frac{n}{2} - r + 16\gamma n,\]
    which is impossible by the property of quadratic function. That is a contradiction.

    \textbf{Case 2:} $a_1\le \gamma^{1/3} n\le c_1$ or $c_1\le \gamma^{1/3} n\le a_1$.
    
    By symmetry, we can assume that $a_1\le \gamma^{1/3} n\le c_1$.
    By inequalities (\ref{eq:c1a2}) and (\ref{eq:vc1}) we have 
    \begin{align}
        \label{c1a2}
        \frac{3n}{4} &\le a_1 + \frac{c_1}{2} + \frac{e(R,A_2)}{a_2} + \frac{e(C_1,R)}{c_1} + 10 \gamma^{2/3} n,\\
        \label{vc1}
        \frac{3n}{4} &\le a_1 + \frac{c_1}{2} + 2a_2 + c_2 + \frac{e(R,C_1)}{c_1} + \gamma^{2/3} n.
    \end{align}
    Note that $e(R'_A,C_1) + e(C_1,R'_A) \le 2\gamma n^2$ by inequality (\ref{eq:4.2}) and the definition of $R'_A$. Combining inequalities (\ref{c1a2}) and (\ref{vc1}), we obtain
    \[\frac{3n}{2} \le 2a_1 + c_1 + 2a_2 + c_2 + \frac{e(R,A_2)}{a_2} + (r-r'_a) + 12\gamma^{2/3} n.\]
    Notice that
    \[n - 1 = a_1 + a_2 + c_1 + c_2 + r,\]
    we have
    \begin{align}
        \label{eq:4.4}
        \frac{n}{2} \le a_1 + a_2 + \frac{e(R,A_2)}{a_2} - r'_a + 12\gamma^{2/3} n.
    \end{align}
    By \cref{cl:rarc} we have

\begin{align}
        \label{eq:add-1}
       \frac{n}{2} \le a_1 + a_2 + r_{ac} + 13\gamma^{2/3} n.
    \end{align} 
By the definitions of $A_2$ and $R_{AC}$, we have $e(G[A_2])$, $e(G[R_{AC}])\le \gamma n^2$.
    Applying \cref{lm:emptyset} we can get
    \begin{align}
        \label{eq:4.5}
        a_2 \le \frac{n}{4} + 21\gamma n,~r_{ac} \le \frac{n}{4} + 21\gamma n.
    \end{align}
    Together with $a_1 \le \gamma^{1/3} n$ and (\ref{eq:add-1}) we have
    \[\frac{n}{4} - 2\gamma^{1/3} n \le a_2,r_{ac} \le \frac{n}{4} + 21\gamma n.\]
    Again by (\ref{eq:4.4}) we have
    \[\frac{e(R,A_2)}{a_2} - r'_a \ge \frac{n}{4} - 1.5\gamma^{1/3} n,\]
    while
    \[r_{ac} = r_a - r'_a \le \frac{n}{4} + 21\gamma n.\]
    It follows that
    \[r_a \le \frac{e(R,A_2)}{a_2} + 2\gamma^{1/3} n.\]
    By \cref{cl:rarc} we have
    \[e(R,A_2) \ge a_2r_a - 2\gamma^{1/3} n^2.\]
    Thus,
    \begin{align}
        \label{eq:4.6}
        e(R_A,A_2) \ge e(R,A_2) - \gamma n^2 \ge a_2 r_a - 3\gamma^{1/3} n^2.
    \end{align}
    
    Now let $A \coloneqq \emptyset$, $B \coloneqq A_2$, $D \coloneqq R_{AC}$, and $C \coloneqq V(G) \setminus (B \cup D)$. We claim that $G$ is $\eta$-extremal with respect to this partition. From (\ref{eq:4.5}), we have
    \[\frac{n}{2} - 42\gamma n \le |C| \le \frac{n}{2} + 4\gamma^{1/3} n.\]
    Clearly the sizes of the partition satisfy the condition of $\eta$-extremal. Recall that $e(B),e(D) \le \gamma n^2$. It remains to check $e(B,C)$, $e(C)$, $e(C,D)$ and $e(D,B)$. Since $B=A_2$ and $D\subseteq R_A$, by (\ref{eq:4.6}) we know that
    \[e(D,B) \ge |B||D| - 3\gamma^{1/3} n^2.\]
    For every vertex $x\in R_{AC}$ we have $vx\notin E(G)$, which implies that
    \[\deg^-(x) \ge \frac{n}{2} - 2\gamma^{1/3} n.\]
    So we have
    \[e(C,D) \ge |D| \cdot \deg^-(x) - e(D) - e(B,D) \ge \frac{n}{2}|D| - 6\gamma^{1/3} n^2 \ge |C||D| - 8\gamma^{1/3} n^2.\]
    Now consider all pairs of vertices $(x,y)\in D\times C$ such that $xy\notin E(G)$. By Ore-degree condition we have
    \begin{align*}
        \frac{3n-3}{4} (|C||D| - 8\gamma^{1/3} n^2) &\le \sum_{\substack{x\in D,y\in C\\xy\notin E(G)}} \deg^+(x) + \deg^-(y)\\
        &\le |C| (e(D,B) + e(D,C) + e(D)) + |D| (e(B,C) + e(C) + e(D,C)).
    \end{align*}
    So
    \[e(B,C) + e(C) \ge |C|\left( |B| + \frac{|C|}{2} \right) - 65\gamma^{1/3} n^2,\]
    which gives lower bounds of $e(B,C)$ and $e(C)$. Therefore we have proved that $G$ is $\eta$-extremal since $\gamma \ll \eta$, a contradiction.

    \textbf{Case 3:} $a_1\ge \gamma^{1/3} n$, $c_1\ge \gamma^{1/3} n$.
    
    By (\ref{eq:c1a1}), we have
    \begin{align}
        \label{eq:4.7}
        \frac{3n}{4} \le \frac{a_1}{2} + \frac{c_1}{2} + \frac{e(R,A_1)}{a_1} + \frac{e(C_1,R)}{c_1} + \gamma^{1/3} n\le \frac{a_1}{2} + \frac{c_1}{2} + r_a+ r_c + 2\gamma^{1/3} n.
    \end{align}
    It follows from (\ref{eq:a1c1})  that
    \[\frac{3n}{4} \le \frac{a_1}{2} + \frac{c_1}{2} + 2(a_2 + c_2) + \frac{e(A_1,R)}{a_1} + \frac{e(R,C_1)}{c_1} + \gamma^{1/3} n.\]
    Taking the sum of above two inequalities we have
    \[\frac{3n}{2} \le a_1 + c_1 + 2(a_2 + c_2) + 2r + 2\gamma^{1/3} n \le 2n - a_1 - c_1 + 2\gamma^{1/3} n,\]
    that is,
    \begin{align}
        \label{eq:4.8}
        a_1 + c_1 \le \frac{n}{2} + 2\gamma^{1/3} n.
    \end{align}
    By (\ref{eq:4.7}) we have
    \begin{align}
        \label{rac-r'}
        r_{ac} - r' = r_a + r_c - r \ge -\frac{n}{4} + \frac{a_1}{2} + \frac{c_1}{2} + a_2 + c_2 - 2\gamma^{1/3} n.
    \end{align}
Dividing (\ref{eq:c1a1}) by $a_1 c_1$, (\ref{eq:a1v}) by $a_1$, and (\ref{eq:vc1}) by $c_1$, and then adding the resulting inequalities, we obtain
    
    \[\frac{9n}{4} \le 2a_1 + 2c_1 + 3a_2 + 3c_2 + (r-r'_a) + (r-r'_c) + 2\gamma^{1/3} n \le 2n + a_2 + c_2 - r'_a - r'_c + 2\gamma^{1/3} n,\]
    that is,
    \begin{align}
        \label{eq:4.9}
        a_2 + c_2 \ge \frac{n}{4} + r'_a + r'_c - 2\gamma^{1/3} n.
    \end{align}
    Substituting to (\ref{rac-r'}) we have
    \begin{align}
        \label{eq:4.10}
        r_{ac} - r' \ge \frac{a_1}{2} + \frac{c_1}{2} + r'_a + r'_c - 4\gamma^{1/3} n.
    \end{align}
    We claim that $r_{ac} \ge n/10$, which will allow us to apply (\ref{eq:vrac}) later. In fact, by (\ref{eq:4.7}) and (\ref{eq:4.8}) we can get
    \[r_a + r_c \ge \frac{n}{2} - 3 \gamma^{1/3} n.\]
    Again noting that $r_a + r_c = r'_a + r'_c + 2r_{ac} \le 3r_{ac} + 3\gamma^{1/3} n$ by (\ref{eq:4.10}), which implies that $r_{ac} \ge n/6 - 2\gamma^{1/3} n$.
    Hence we may use (\ref{eq:vrac}) and get
    \[r_{ac} + r - \frac{n}{2} \le 20\gamma n.\]
    It implies that
    \begin{align*}
        a_1 + a_2 + c_1 + c_2 &= n - r - 1 \ge \frac{n}{2} + r_{ac} - 21\gamma n\\
        &\ge \frac{n}{4} + \frac{a_1}{2} + \frac{c_1}{2} + a_2 + c_2 - 3\gamma^{1/3} n,
    \end{align*}
    where the last inequality holds by (\ref{rac-r'}). Hence we have
    \begin{align}
        \label{eq:4.11}
        \frac{n}{2} - 6\gamma^{1/3} n \le a_1 + c_1 \le \frac{n}{2} + 2\gamma^{1/3} n.
    \end{align}
    Again by (\ref{eq:4.7}) we have
    \begin{align*}
        \frac{3n}{4} &\le \frac{1}{2} (n - a_2 - c_2 - r) + r_a + r_c + 2\gamma^{1/3} n\\
        &\le \frac{1}{2} (n - a_2 - c_2 - r) + (r - r'_c) + (r - r'_a) + 2\gamma^{1/3} n\\
        &= \frac{3}{2}r + \frac{n}{2} - \frac{1}{2} (a_2 + c_2 - r'_a - r'_c) - \frac{3}{2} (r'_a + r'_c) + 2\gamma^{1/3} n\\
&\overset{(\ref{eq:4.9})}{\le} 
         \frac{3}{2}r + \frac{3n}{8} - \frac{3}{2} (r'_a + r'_c) + 3\gamma^{1/3} n,
    \end{align*}
    that is,
    \[r \ge \frac{n}{4} + (r'_a + r'_c) - 2\gamma^{1/3} n.\]
    Hence combining (\ref{eq:4.9}) and (\ref{eq:4.11}) we have
    \[\frac{n}{4} - 2\gamma^{1/3} n \le r \le \frac{n}{4} + 8\gamma^{1/3} n,\]
    and
    \[\frac{n}{4} - 2\gamma^{1/3} n \le a_2 + c_2 \le \frac{n}{4} + 8\gamma^{1/3} n.\]
    Let $A \coloneqq A_1$, $B \coloneqq A_2 \cup C_2$, $C \coloneqq C_1$, and $D \coloneqq R \cup \{v\}$. It remains to prove that $G$ is $\eta$-extremal with respect to this partition. In fact, we can obtain stronger results about $a_1$ and $c_1$ beyond (\ref{eq:4.11}). Let us consider the complement of $G[A]$ in the complete directed graph of order $a_1$, named $\overline{G}[A]$. Since $G$ is an oriented graph, $\overline{G}[A]$ must contain a tournament, and hence a Hamilton path $v_1v_2\ldots v_{a_1}$. Now we may use Ore-degree condition and get
    \begin{align}
        \label{eq:4.12}
        \frac{3n-3}{4} (a_1 - 1) \le \sum_{i=1}^{a_1-1} \deg^+(v_i) + \deg^-(v_{i+1}) \le \sum_{x\in A} \deg(x) \le 2\binom{a_1}{2} + a_1 (|B| + |D|) + 2 \gamma n^2.
    \end{align}
    It implies that
    \[a_1 \ge \frac{n}{4} - 17\gamma^{1/3} n.\]
    By symmetry and (\ref{eq:4.11}) we conclude that
    \[\frac{n}{4} - 17\gamma^{1/3} n \le a_1,c_1 \le \frac{n}{4} + 19\gamma^{1/3} n.\]
    Note that by (\ref{eq:4.10}) we have
    \[r_{ac} \ge (a_1 + c_1)/2 - 4\gamma^{1/3} n \ge n/4 - 7\gamma^{1/3} n.\]
    A more precise analysis of $\deg(x)$ in (\ref{eq:4.12}) yields
    \begin{align}
        \label{eq:5.19}
        \frac{3n}{4} a_1 \le 2 e(A) + e(A,B) + e(D,A) + 6\gamma^{1/3} n^2.
    \end{align}
    To see this, by the definition of the partition and (\ref{eq:4.2}) we have
    \begin{align*}
        &e(B,A) = e(A_2,A_1) + e(C_2,A_1) \le 2 \gamma n^2,\\
        &e(C,A) + e(A,C) = e(C_1,A_1) + e(A_1,C_1) \le 2 \gamma n^2,\\
        &e(A,D) = e(A_1,R_{AC}) + e(A_1,R\setminus R_{AC}) \le \gamma n^2 + a_1\cdot 15\gamma^{1/3} n \le 5\gamma^{1/3} n^2.
    \end{align*}
    Note that $2e(A) + e(A,B) + e(D,A) \le a_1(|A| + |B| + |D|)$. 
The terms $e(A)$, $e(A,B)$, and $e(D,A)$ satisfy the condition of required for $\eta$-extremal by the choice of $\gamma \ll \eta$. By symmetry, it can be similarly demonstrated that $e(C)$, $e(B,C)$, and $e(C,D)$ also meet these conditions. It remains to verify $e(B,D)$. To this end, consider all pairs of vertices $(x,y)\in A\times D$ such that $xy\notin E(G)$, then we have
    \[\frac{3n-3}{4} (|A||D| - 5\gamma^{1/3} n^2) \le \sum_{x\in A,y\in D} \deg^+(x) + \deg^-(y) \le |D| \sum_{x\in A} \deg^+(x) + |A| \sum_{y\in D} \deg^-(y).\]
    For the right side, we have
    \[\sum_{x\in A} \deg^+(x) \le |A|^2/2 + |A||B| + 6\gamma^{1/3} n^2,\]
    and
    \[\sum_{y\in D} \deg^-(y) \le |C||D| + e(B,D) + 5\gamma^{1/3} n^2 + e(D).\]
    We also have
    \[e(D) \le e(R_{AC}) + e(R\setminus R_{AC},R) + e(R,R\setminus R_{AC}) \le \gamma n^2 + r \cdot 15\gamma^{1/3} n \le 5\gamma^{1/3} n^2.\]
    Summarizing the above inequalities, we conclude that
    \[e(B,D) \ge |A|n/8 - 55\gamma^{1/3} n^2.\]
    This bound holds for $e(D,B)$ symmetrically. Recall that $e(A,C) \le \gamma n^2$ by (\ref{eq:4.2}). Therefore, $G$ is $\eta$-extremal, which is a contradiction.
\end{proof}

\begin{proof}[Proof of (\ref{eq:c1a1})]
    Consider all pairs of vertices $(x,y) \in C_1 \times A_1$ with $xy \notin E(G)$. By (\ref{eq:4.1}) and Ore-degree condition, we have
    \[\frac{3n-3}{4} (a_1c_1 - \alpha_1
    n^2) \le \sum_{\substack{x\in C_1,y\in A_1\\xy \notin E(G)}} \deg^+(x) + \deg^-(y).\]
    For the right side, we have
    \[\sum_{\substack{x\in C_1,y\in A_1\\xy \notin E(G)}} \deg^+(x) \le a_1 \sum_{x\in C_1} \deg^+(x) \le a_1 \binom{c_1}{2} + a_1 \cdot e(C_1,R) + a_1 (\alpha_1 n^2 + c_2 \cdot \gamma n),\]
    since there are at most $\alpha_1 n^2$ edges from $C_1$ to $A$ by (\ref{eq:4.1}), and at most $c_2 \cdot \gamma n$ edges from $C_1$ to $C_2$ by the definition of $C_2$.
    Similarly we have
    \[\sum_{\substack{x\in C_1,y\in A_1\\xy \notin E(G)}} \deg^-(y) \le c_1 \binom{a_1}{2} + c_1 \cdot e(R,A_1) + c_1 (\alpha_1 n^2 + a_2 \cdot \gamma n).\]
    Hence we have
    \[\frac{3n}{4} a_1 c_1 \le c_1\binom{a_1}{2} + a_1\binom{c_1}{2} + c_1\cdot e(R,A_1) + a_1\cdot e(C_1,R) + \left(a_1 c_1 + 3\alpha_1 n^3 + (a_1 c_2 + c_1 a_2) \gamma n\right).\]
    Note that $a_1 c_1 \le n^2$, $\alpha_1 \ll \gamma$, and $a_1 c_2 + c_1 a_2 \le n^2/2$ by the partition of $V(G)$, the last term of the above inequality is less than $\gamma \cdot n^3$. Therefore, we have proved (\ref{eq:c1a1}).
\end{proof}

\begin{proof}[Proof of (\ref{eq:vra'})]
    Note that for every vertex $x\in R'_A$, $x$ is not adjacent to $v$. So we have
    \[\frac{3n-3}{2} r'_a \le \sum_{x\in R'_A} \deg(x) + \deg(v) \le (n - r - 1) r'_a + \sum_{x\in R'_A} \deg(x).\]
    By the definition of $R'_A$ and (\ref{eq:4.2}) we have
    \[e(C_1,R'_A) \le \gamma n^2, ~e(R'_A,C_1) \le \gamma n^2.\]
    Hence we can get
    \[(\frac{n-1}{2} + r) r'_a \le (n-1 - c_1) r'_a + 2\gamma n^2.\]
Therefore,
    \[r'_a (c_1 + r - \frac{n}{2}) \le 2\gamma n^2,\]
as desired.
\end{proof}

\begin{proof}[Proof of (\ref{eq:a1v})]
    For every vertex $x \in A_1$, $xv\notin E(G)$. So we have
    \[\frac{3n-3}{4} a_1 \le \sum_{x\in A_1} \deg^+(x) + d^-(v) \le \sum_{x\in A_1} \deg^+(x) + a_1 (c_1 + c_2).\]
    By (\ref{eq:4.2}) we have
    \[e(A_1,C_1) \le \frac{\gamma n^2}{2}.\]
    Hence we can get
    \[\frac{3n}{4} a_1 \le \binom{a_1}{2} + a_1 (a_2 + c_2) + e(A_1,R) + a_1 (c_1 + c_2) + \gamma n^2,\]
    as desired.
\end{proof}
We omit the proofs of the remaining inequalities, as they can be obtained using the same approach discussed above.

\section{Proof of Covering Lemma}\label{s5}
\begin{proof}[Proof of Lemma~\ref{lem:almostcover}]
Let $M,M'\in \mathbb{N}$ be additional constants such that
\[1/n \ll 1/M \ll 1/M' \ll \epsilon \ll d \ll \eta \ll  1.\]
Let $G$ be an oriented graph on $n$ vertices with $\deg^{+}(x)+\deg^{-}(y)\ge (3n-3)/4$ for any $xy\notin E(G)$, and let $X\subseteq V(G)$ with $|X|\le dn$. Set $G_X \coloneqq G[V(G)\setminus X]$ the subgraph of $G$ induced by $V(G)\setminus X$. For every pair of distinct vertices $x,y$ with $xy\notin E(G_X)$ we have $\deg^+(x) + \deg^-(y) \ge (3n-3)/4 - 2dn$. And we have $\delta^0(G_X) \ge n/8 - dn$ by \cref{le:delta0}.

We first apply the Diregularity lemma~\ref{lem:reg} to $G_X$ with parameters $(\epsilon^2/3,d,M')$ to obtain a partition $V_0,V_1,\ldots,V_k$ of $V(G_X)$ with $k\ge M'$. Let $R$ be the reduced oriented graph with parameters $(\epsilon^2/3,d)$ given by Lemma~\ref{lem:reduceR}. Let $G^*$ be the pure oriented graph. By \cref{lem:reduceR} we have
\begin{align}
    \delta^0(R) \ge \frac{k}{8} - 2dk.\label{equ:reduce-semideg}
\end{align}
And for any $x,y\in V(R)$ with $xy\notin E(R)$, we have
\begin{align}
    \deg_R^+(x)+\deg_R^-(y) \ge
    \frac{3k}{4} -3dk.\label{equ:reduce-dgree}
\end{align}
We divide the proof into two cases by whether there is a set $S\subset [k]$  with $|N_R^+(S)|<|S|$.\\

Case 1: For every $S\subset [k]$, we have $|N_R^+(S)|\ge |S|$ .\\

We claim that $R$ contains a $1$-factor. In fact, let us consider the following bipartite graph $H$: the vertex classes $A,B$ of $H$ are both copies of $V(R)$ and we connect a vertex $a\in A$ to $b\in B$ in $H$ if there is a directed edge from $a$ to $b$ in $R$. The graph $H$ admits a perfect matching by Hall's matching theorem, which corresponds to a $1$-factor in $R$.

Let $G'_S$ be the subgraph of $G^*$ obtained from \cref{prop:superregular} by letting $S$ be a $1$-factor of $R$ (and hence $D=2$). More precisely, $V(G'_S)$ is the union of $V'_i$, each of which is a subcluster of $V_i$ by deleting exactly $4\epsilon |V_i|$ vertices. And for every $ij\in E(S)$, $G'_S[V'_i,V'_j]$ is $(\sqrt{\epsilon}, d-8\epsilon)$-super-regular. Note that all $V'_i$ have the same order, and every balanced blow-up of a directed cycle is Hamiltonian. Applying \cref{lem:blowup} we can get a $1$-factor $S'$ of $G'_S$ that consists of at most $k/3 \le M/3$ cycles. The number of uncovered vertices in $V(G)\setminus X$ is at most $|V_0| + \sum_{i=1}^{k} 4\epsilon|V_i| \le 5\epsilon n$. Removing one edge from each cycle in $S'$ we get a family of paths that satisfies the first alternative of this lemma.\\

Case 2: There is a set $S\subset [k]$ with $|N_R^+(S)|<|S|$.\\

Let us define
\[A_R\coloneqq S\cap N^+_R(S), ~B_R\coloneqq N^+_R(S)\setminus S, ~C_R\coloneqq [k]\setminus(S\cup N^+_R(S)), ~D_R\coloneqq S\setminus N^+_R(S).\]
Let $A\coloneqq\cup_{i\in A_R}V_i$, and define $B, C, D$ similarly. We further put all the vertices of $X \cup V_0$ into $A$. In order to show that $G$ is $\eta$-extremal, let us first prove, in the following claim, that the sizes of these four vertex subsets are appropriate.

\begin{cl}\label{cl:1}
$|B|= (1/4 \pm O(d)) n$,
$|D| = (1/4 \pm O(d)) n$. 
\end{cl}

\begin{proof}
    Let us first prove that $|D_R| \ge 2$ so that we may apply the degree condition (\ref{equ:reduce-dgree}). Suppose that $|D_R| \le 1$. By definition we have $|A_R| + |B_R| < |A_R| + |D_R|$, and hence $|B_R| = 0$, $|D_R| = 1$. (\ref{equ:reduce-semideg}) implies that neither $A_R$ nor $C_R$ are empty, so we may choose two vertices $x\in A_R$ and $y\in C_R$ such that
    \[\frac{3k}{4} - 3dk \le \deg^+_R(x) + \deg^-_R(y) \le \frac{|A_R|}{2} + \frac{|C_R|}{2} \le \frac{k}{2},\]
    which is a contradiction.

    Now we have $|D_R| \ge 2$. Since $D_R$ is an independent set in $R$, for any two distinct vertices $x,y\in D_R$, we have
    \[\frac{3k}{2} - 6dk \le \deg_R(x) + \deg_R(y) \le 2(|A_R| + |B_R| + |C_R|) = 2(k - |D_R|).\]
    And hence
    \[|B_R| < |D_R| \le \frac{k}{4} + 3dk.\]
    Similarly we may choose two vertices $x\in A_R$ and $y\in C_R$ such that
    \[\frac{3k}{4} -3dk \le  \deg^+_R(x) + \deg^-_R(y) \le \frac{|A_R|}{2} + 2|B_R| + \frac{|C_R|}{2} < \frac{k}{2} + |B_R|.\]
    And hence
    \[|B_R| > \frac{k}{4} - 3dk.\]
    Therefore, we have
    \[\frac{k}{4} - 3dk < |B_R| < |D_R| \le \frac{k}{4} + 3dk.\]
    The claim follows after restoration to the original graph.
\end{proof}

The definition directly implies that $e_R(A_R, C_R) =
e_R(A_R, D_R) = e_R(D_R, C_R) = e_R(D_R)=0$. Since $R$ has parameters $(\epsilon^2/3,d)$, Lemma~\ref{lem:reduceR} (1) and (2) imply that we have
\begin{align}
    e(A,C),e(A,D),e(D,C), e(D)\le 3dn^2.\label{equ:edge less}
\end{align}
It remains to prove the following claim.

\begin{cl}\label{cl:2}
\settasks{
  label = (\arabic*), 
  label-width = 1.5em, 
  item-indent = 2em,   
  before-skip = 18pt,   
  after-skip = 6pt     
}
\begin{tasks}(2)
\task $e(A,B) \ge |A||B| -O(\eta)n^2$;\label{cl:2-1}
\task $e(B,C) \ge |B||C| -O(\eta)n^2$; \label{cl:2-2}
\task $e(D,A) \ge |A||D| - O(\eta)n^2$;\label{cl:2-3}
\task $e(C,D) \ge |C||D| - O(\eta)n^2$;\label{cl:2-4}
\task $e(B,D) \ge |A|n/8 - O(\eta)n^2$; \label{cl:2-5}
\task $e(D,B) \ge |C|n/8 - O(\eta)n^2$;\label{cl:2-6}
\task  $e(A) \ge |A|^2/2 - O(\eta)n^2$;\label{cl:2-7}
\task $e(C) \ge |C|^2/2 - O(\eta)n^2$. \label{cl:2-8}
\end{tasks}

\end{cl}

\begin{proof}
    Inequalities \ref{cl:2-1}, \ref{cl:2-3}, \ref{cl:2-5} and \ref{cl:2-7} are trivially valid whenever $|A| = O(\eta) n$. Let us prove these four inequalities under the assumption that $|A| \ge \eta n$. The key point of proof is to make use of those vertices whose outdegree or indegree is close to a minimal extremal example. So in the following proof, we will use the outdegree of $A\cup D$ and the indegree of $C\cup D$. Consider all pairs of vertices $(x,y)$ with $x\in A$, $y\in D$ and $xy\notin E(G)$, we have
    \begin{align*}
        \sum_{\substack{x\in A,y\in D\\xy\notin E(G)}} \deg^+(x) + \deg^-(y) \ge \frac{3n-3}{4} (|A||D| - e(A,D)) \ge \frac{3n}{4}|A||D| - 3dn^3,
    \end{align*}
    and
    \begin{align*}
        \sum_{\substack{x\in A,y\in D\\xy\notin E(G)}} \deg^+(x) + \deg^-(y) &\le |D| \sum_{x\in A} \deg^+(x) + |A| \sum_{y\in D} \deg^-(y)\\
        &\le |D| (e(A) + e(A,B)) + |A| (e(B,D) + e(C,D)) + 6dn^3.
    \end{align*}
    It follows that
    \begin{align}
        \frac{3n}{4} |A||D| \le |D| (e(A) + e(A,B)) + |A| (e(B,D) + e(C,D)) + 9dn^3.\label{equ:5.3ad}
    \end{align}
    Similarly we have
    \begin{align}
        \frac{3n}{4} |C||D| \le |C| (e(D,A) + e(D,B)) + |D| (e(C) + e(B,C)) + 9dn^3,\label{equ:5.3dc}
    \end{align}
    and
    \begin{align}
        \frac{3n}{4} |D|^2 \le |D| (e(D,A) + |B||D| + e(C,D)) + 6dn^3.\label{equ:5.3dd}
    \end{align}
    We may combine (\ref{equ:5.3ad}) and (\ref{equ:5.3dc}) and get
    \begin{align}
        \nonumber \frac{3n}{2} |A||C||D| &\le |A||C| (e(D,A) + |B||D| + e(C,D))\\
        &~~+ |A||D| (e(C) + e(B,C)) + |C||D| (e(A) + e(A,B)) + 18dn^4.\label{equ:5.3adc}
    \end{align}

    We first prove \ref{cl:2-1} and \ref{cl:2-7}. If $|C| \ge \eta n$, then by (\ref{equ:5.3adc}) we have
    \[\frac{3n}{2} \le (|A| + |B| + |C|) + \left(\frac{|C|}{2} + |B|\right) + \frac{e(A) + e(A,B)}{|A|} + O(\eta)n,\]
    that is
    \[\frac{n}{2} \le \frac{|C|}{2} + \frac{e(A) + e(A,B)}{|A|} + O(\eta)n.\]
    On the other hand, it trivially holds that $e(A) \le |A|^2/2$ and $e(A,B) \le |A||B|$. Hence noting that $|A| + |C| = n/2 \pm O(d) n$ and $|B| = n/4 \pm O(d) n$ we must have
    \[e(A) \ge \frac{|A|^2}{2} - O(\eta)n^2,~~~e(A,B) \ge |A||B| - O(\eta)n^2.\]
    If $|C| < \eta n$, then by (\ref{equ:5.3ad}) we have
    \[\frac{3n}{4} \le |B| + \frac{e(A) + e(A,B)}{|A|} + O(d) n.\]
    Similarly, considering the size of $A,B,D$ we can prove \ref{cl:2-1} and \ref{cl:2-7}.

    To prove \ref{cl:2-3}, by (\ref{equ:5.3dd}) we can get
    \[\frac{3n}{4} \le |B| + |C| + \frac{e(D,A)}{|D|} + O(d) n.\]
    Since $|A| + |B| + |C| \le 3n/4 + O(d) n$, we have
    \[e(D,A) \ge |A||D| - O(d)n^2 \ge |A||D| - O(\eta)n^2.\]

    To prove \ref{cl:2-5}, by (\ref{equ:5.3ad}) we can get
    \[\frac{3n}{4} \le \frac{|A|}{2} + |B| + |C| + \frac{e(B,D)}{|D|} + O(\eta) n.\]
    Hence we have
    \[e(B,D) \ge \frac{|A||D|}{2} - O(\eta)n^2 \ge \frac{|A|n}{8} - O(\eta)n^2.\]

    By symmetry we can prove \ref{cl:2-2}, \ref{cl:2-4}, \ref{cl:2-6} and \ref{cl:2-8} with respect to $|C|$ via a similar argument. Hence the claim follows.
\end{proof}

Therefore, $G$ is $\eta$-extremal by our choice of $d\ll \eta$.
\end{proof}

\section{Extremal case: proof of Theorem~\ref{thm:extremal}}\label{section-extremal}
The main idea of this proof is that of \cite{keevash2009exact}, as following. By exploiting Ore-degree condition, we first demonstrate that $G$ must exhibit a structure analogous to the extremal case described in Section~\ref{sec:extremal}. Through a sequence of further claims, we then either find a Hamilton cycle directly or deduce additional structural constraints that force $G$ to align even more precisely with the extremal example. This leads to a contradiction, since the extremal example inherently fails to satisfy Ore-degree condition while $G$ must satisfy this condition by hypothesis.

The proofs need several modifications due to the difference of degree condition and consequently the more general structure of extremal graphs. One significant thing is that the order of $A$ does make sense. More precisely, for instance, we have $e(A,B) > |A||B| - O(\eta)n^2$, that all but a tiny proportion of possible edges from $A$ to $B$ must exist. However, the $O(\eta)n^2$-term cannot be ignored whenever $A$ contains only a few vertices. We will use a appropriate regime, in the following proof, to determine whether $A$ is sufficiently large or not.

By symmetry we may assume that $|A| \le |C|$, since one can reverse the direction of all edges. Let $\eta_1,\mu,\lambda$ be additional constants such that
\[1/n \ll \eta\ll \eta_1 \ll \mu \ll \lambda \ll 1.\]
We split the proof into two cases with respect to the order of $A$.

\subsection{When $|A| \ge \sqrt{\mu}\cdot n$}
\label{subsec:7.1}
A simple observation shows that, in the non-degenerate case where each partition is sufficiently large, the extremal structure is essentially close to that of~\cite{keevash2009exact}. Hence, it is expected that most of the analysis would be analogized to Ore-degree condition. More specifically, we refer to the same concept of \emph{cyclic}, \emph{acceptable}, and \emph{good} vertices. We first prove Claims~\ref{cl:cyclic} and~\ref{cl:mindeg-ac} using slightly numerical computing. Claims~\ref{cl:3} and~\ref{cl:6} state some properties of different types of vertices, whose proofs can be achieved simply by adjusting some relative parameters. \cref{cl:4} provides a way to find a Hamilton cycle whenever $|B| = |D|$, while~\cref{cl:7} characterizes the edges that are not expected to appear in extremal graphs. In the proof of these two claims, we overcome the main difference of degree conditions and extremal structures by some structural or numerical analysis. For clarity and completeness, the proofs of~\cref{cl:3} and~\ref{cl:6} will be omitted and put in Appendix.

To facilitate our analysis, we define the notation $A\coloneqq P(1)$, $B\coloneqq P(2)$, $C\coloneqq P(3)$, $D\coloneqq P(4)$.
For a vertex $x\in G$, we introduce the compact notation $x$ has property $W:A_{*}^{*}B_{*}^{*}C_{*}^{*}D_{*}^{*}$ as follows. 
The notation starts with some $W\in \{A,B,C,D\}$, namely the set that $x$ belongs to.
Next, for each of $A,B,C,D$, the superscript symbol refers to the intersection with $N^{+}(x)$ and the subscript to the intersection with $N^{-}(x)$. 
Specifically, for $Z\in \{A,B,C,D\}$, a symbol `$>\alpha$' denotes the intersection of size at least $\alpha |Z|-O(\eta_1)n$, while a symbol `$<\alpha$' denotes the intersection of size at most $\alpha |Z|+O(\eta_1)n$.
The absence of a symbol means that there is no restriction on the intersection,
and we can omit any of $A,B,C,D$ if it has no superscript or subscript. 
For instance, to say $x$ has property $B: A^{>1}C^{<1/2}_{<1/3}$ means that $x\in B$, $|N^+(x)\cap A|>|A|-O(\eta_1)n$, $|N^+(x)\cap C|<|C|/2+O(\eta_1)n$ and $|N^-(x)\cap C|<|C|/3+O(\eta_1)n$.

We now give the definitions of cyclic and acceptable for a vertex in $G$. We say that a vertex $x$ is \emph{cyclic} if it satisfies 
$$P(i): P(i+1)^{>1}P(i-1)_{>1}$$
for some $i\le 4$ (counting modulo $4$). The following fact implies that $G$ has at most $O(\eta_1)n$ non-cyclic vertices.
\begin{claim}\label{cl:cyclic}
    The number of non-cyclic vertices in $G$ is at most $O(\eta_1)n$.
\end{claim}
\begin{proof}
    Let $t$ be the number of non-cyclic vertices in $B$. According to the definition of cyclic vertices, we have 

   \begin{align}
    e(B,C)=\sum_{\substack{v\in B\\v~\text{is cyclic}}}\deg_C^+(v)+\sum_{\substack{v\in B\\v~\text{is non-cyclic}}}\deg_C^+(v)
    &\le (|B|-t)|C|+t(|C|-\Omega(\eta_1)n).\nonumber
\end{align}
On the other hand, $e(B,C)\ge |B||C|-O(\eta)n^2$ since $G$ is $\eta$-extremal. Then we get $t=O(\eta_1)n$. By the same argument, the number of non-cyclic vertices in $A$, $C$, $D$ is at most $O(\eta_1)n$ respectively. This concludes the proof.
\end{proof}

In the following claim, we will show that the subgraph induced by $A$ or $C$ is quite close to a regular tournament. We would like to mention that the same result trivially holds for the Dirac-type case.

\begin{claim}\label{cl:mindeg-ac}
    For any cyclic vertex $x\in A$, it holds that
    \[\deg^+_{A}(x) \ge \frac{|A|}{2} - O(\eta_1)n.\]
    For any cyclic vertex $x\in C$, it holds that
    \[\deg^-_{C}(x) \ge \frac{|C|}{2} - O(\eta_1)n.\]
\end{claim}

\begin{proof}
    By symmetry it suffices to prove the case for $A$. Since $x$ is cyclic, the number of vertices $y\in D$ with $xy\notin E(G_2)$ is at least $|D| - O(\eta_1)n$. We have
    \[\sum_{y\in D,xy\notin E(G)} \deg^+(x) + \deg^-(y) \ge (|D| - O(\eta_1)n) \cdot \frac{3n-3}{4}.\]
    For the left side we have
    \begin{align*}
        \sum_{y\in D,xy\notin E(G)} \deg^+(x) + \deg^-(y) &\le |D| \cdot \deg^+(x) + \sum_{y\in D} \deg^-(y)\\
        &\le |D| \cdot \deg^+(x) + \frac{n^2-|A|n}{8} + O(\eta_1)n^2.
    \end{align*}
    Hence we have
    \[\deg^+(x) \ge \frac{n}{4} + \frac{|A|}{2} - O(\eta_1)n,\]
    and so
    \[\deg^+_{A}(x) \ge \frac{|A|}{2} - O(\eta_1)n.\]
\end{proof}

We say that a vertex is \emph{acceptable} if it has one of the following properties:
\begin{itemize}
        \item $A: B^{>\mu}D_{>\mu}$, $A: A^{>\mu}D_{>\mu}$, $A: A_{>\mu}B^{>\mu}$, $A: A_{>\mu}^{>\mu}$;
        \item $B: A_{>\mu}C^{>\mu}$, $B: A_{>\mu}D^{>\mu}$, $B: C^{>\mu}D_{>\mu}$, $B: D_{>\mu}^{>\mu}$;
        \item $C: B_{>\mu}D^{>\mu}$, $C: B_{>\mu}C^{>\mu}$, $C: C_{>\mu}D^{>\mu}$, $C: C_{>\mu}^{>\mu}$;
        \item $D: A^{>\mu}C_{>\mu}$, $D: A^{>\mu}B_{>\mu}$, $D: B^{>\mu}C_{>\mu}$, $D: B_{>\mu}^{>\mu}$.
\end{itemize}

In other words, a vertex is acceptable if it has a significant outneighborhood in one of its two out-classes and a significant inneighborhood in one of its two in-classes, where `out-classes'
and `in-classes' are to be understood with reference to the extremal oriented graph in Section~\ref{sec:extremal}.
For example, the out-classes of $A$ are $A$ and $B$, and its in-classes are $A$ and $D$. 
We will also call an edge \emph{acceptable} if it is of the type allowed in the extremal oriented graph (so, for example, an edge from $A$ to $B$ is acceptable but an edge from $B$ to $A$ is not). Note that a cyclic vertex
is also acceptable.

In what follows, we will carry out $O(\eta_1)n$ reassignments of vertices between the sets $A$, $B$, $C$ and $D$ (and a similar number of `path contractions' as well). After each sequence of such reassignments it is understood that the hidden constant in the $O(\eta_1)n$-notation of the definition of an acceptable/cyclic vertex is increased.

\begin{claim}\label{cl:3}
By reassigning non-cyclic vertices to $A,B,C,$ or $D$, we can arrange every vertex of $G$ to be acceptable.
\end{claim}

In proving the following claims, the contraction of paths will play a crucial role. Let $P$ be a path in $G$ that starts from $p_1$ to $p_2$ lie in the same class $P(i)$, the \emph{contraction}
of $P$ yields the following oriented graph $H$: we add a new vertex $p$ to the class $P(i)$ and remove
(the vertices of) the path $P$ from $G$. The new vertex $p$ remains the inneighbors of $p_1$ in $P(i-1)$ and the outneighbors of $p_2$ in $P(i+1)$, i.e.~$N^+_H(p)=N^+_G(p_2)\cap P(i+1)$ and $N^-_H(p)=N_G^-(p_1)\cap P(i-1)$. Observe that every cycle in  $H$ containing $p$ corresponds to a cycle in $G$ containing $P$. The paths $P$ utilized in the proof of Claim \ref{cl:4} will have initial and final vertices in the same class and will be \emph{acceptable}, meaning that every edge on $P$ is acceptable.
Note that each such acceptable path $P$ must be \emph{BD-balanced}, which means that if we delete the initial vertex of $P$ we are left with a path that meets $B$ and $D$ the same number of times. This may be seen from the observations that visits of $P$ to $B\cup D$ alternate between $B$ and $D$, and if the path is in $A$ and then leaves, it must visit $B$ and $D$ an equal number of times
before returning to $A$ (and similarly for $C$). Notice that if $|B| = |D|$ and we contract a $BD$-balanced path, then the resulting digraph will still satisfy $|B| = |D|$. The `moreover' part
of the following claim is later utilized in the proof to transform a graph with 
$|B| = |D| + 1$ into  one with $|B| = |D|$
under specific situations.

\begin{claim}\label{cl:4}
If $|B|=|D|$, $|C| \ge |A| = \Omega(\sqrt{\mu}) n$, and every vertex is acceptable, then $G$ has a Hamilton cycle. Moreover, the assertion remains valid under a slightly relaxed condition: $|B|=|D|$, and there exists a vertex $x$ such that all vertices except $x$ are acceptable, there is at least one acceptable edge directed toward $x$ and one acceptable edge directed away from $x$.
\end{claim}
\begin{proof}
    For clarity, we divide the proof into five steps.
    
    \textit{Step~$1$}. Contract some suitable paths into vertices such that the resulting digraph consists entirely of cyclic vertices. Let $v_1,\ldots,v_t$ be the vertices that are acceptable but not cyclic. For each $v_i$, choose a cyclic outneighbor $v_i^+$ and a cyclic inneighbor $v_i^-$ so that all of these vertices are distinct, and so that the edges $v_iv_i^+$ and $v_i^-v_i$ are acceptable. This can be done since $t=O(\eta_1)n$. Let $P'_i$ be a path of length at most 3 starting at $v_i^+$ and ending at a cyclic vertex that lies in the same class as $v_i^-$, and where the successive vertices lie in successive classes, that is, the successor of a vertex $x \in V (P)\cap P(i)$ lies in $P(i+1)$. For example, if $v_i$ has the first of the acceptable properties of a vertex in $A$, then we can choose $v_i^+\in B,v_i^-\in D$ and so $P'_i$ would have its final vertex in $D$. Also, if $v_i^-$ and $v_i^+$ lie in the same class, then $P'_i$ consists of the single vertex $v_i^+$. Note that the paths $P'_i$ can be chosen to be disjoint. Let $P_i\coloneqq v_i^{-}v_iv_i^+P'_i$. Then $P_i$ are acceptable and thus $BD$-balanced. Let $G_1$ be the oriented graph obtained from $G$ by contracting the paths $P_i$. Then every vertex of $G_1$ is cyclic (by changing the constant involved in the $O(\eta_1)$-notation in the definition of a cyclic vertex if necessary). Moreover, the sets $A, B, C, D$ in $G_1$ still satisfy $|B| = |D|$, and we still assume that $|A| \le |C|$ by symmetry.

    \textit{Step~$2$}. Reserve some matchings between $B$ and $D$, which depends on the size of these vertex classes. Since there are at most $O(\eta_1)n$ vertices that are not cyclic, it still holds, for $G_1$, that $e(B,D) > |A|n/8 - O(\eta_1)n^2$ and $e(D,B) > |C|n/8 - O(\eta_1)n^2$. If $|D| \le |A|$ then we reserve no matchings. If $|A| < |D| \le |C|$, then we reserve one matching $M_1$ of size $|D|-|A|$ in $E(D,B)$. This can always be achieved by applying the bipartite Tur\'{a}n number of matchings, which implies that the maximum matching of $G_1[D,B]$ is at least
    \[\frac{|C|n/8 - O(\eta_1)n^2}{|D|} \ge \frac{|C|}{2} - O(\eta_1)n = \frac{n}{4} - \frac{|A|}{2} - O(\eta_1)n \ge |D| - |A|.\]
    Here the last inequality holds since we assume that $|A| = \Omega(\sqrt{\mu})n$. If $|C| < |D|$, then we reserve a matching $M_1$ of size $|D|-|A|$ in $E(D,B)$, and a matching $M_2$ of size $|D|-|C|$ in $E(B,D)$. Note that we have $|D|-|C| \le |D|-|A| = O(\eta_1)n$ whenever $|C| < |D|$. We can find $M_1$ and $M_2$ that share no common vertices.

    \textit{Step~$3$}. Contract some suitable paths so that $|A|=|B|=|C|=|D|$. To begin with, if $|C| > |D|$, then we claim that there is a directed path $P_C$ of length $|C|-|D|$ in $G_1[C]$. In fact, since at most $O(\eta_1)n$ short paths are contracted in the first step, $G_1[C]$ must contain a subgraph whose minimum indegree is at least $|C|/2 - O(\eta_1)n$ by \cref{cl:mindeg-ac}. Applying  \cite[Theorem 4]{cheng2024length} there is a directed path of length $1.5(|C|/2 - O(\eta_1)n) \ge |C|-|D|$. Similarly we can find such a path $P_A$ of length $|A|-|D|$ whenever $|A| > |D|$. Let $G_2$ be the oriented graph obtained from $G_1$ by contracting $P_C$ and $P_A$. Clearly every vertex of $G_2$ is cyclic.

    When $|A| < |D|$, let us find a $BCD$-balanced path $P_1$ through the matching $M_1$. We denote the edges of $M_1$ to be $y_1x_1,y_2x_2,\ldots,y_sx_s$, where $s=|D|-|A|$, $x_i\in B$ and $y_i\in D$. Choose a vertex $x_0\in B\setminus V(M_1)$. For every $1\le i\le s$, we can find a vertex $z_i\in C$ without repetition such that $x_{i-1}z_i,z_iy_i\in E(G_2)$. This is because every vertex is cyclic and $s = |D|-|A| \le n/4 - \Omega(\sqrt{\mu})n$. Similarly we can find such a path $P_2$ whenever $|C| < |D|$. Let $G_3$ be the oriented graph obtained from $G_2$ by contracting $P_1$ and $P_2$. Then we have $|A|=|B|=|C|=|D|$.

    \textit{Step~$4$}. Find a Hamilton cycle in $G_3$. Recall that every vertex of $G_2$ is cyclic, and $G_3$ has $\Omega(\sqrt{\mu})n$ vertices. For any vertex $x\in P(i)$, $x$ has at least $|P(i+1)| - O(\eta_1/\sqrt{\mu})n$ outneighbors in $P(i+1)$, and at least $|P(i-1)| - O(\eta_1/\sqrt{\mu})n$ inneighbors in $P(i-1)$.
    
    Let $G'_3$ be the underlying graph corresponding to the set of edges oriented from $P(i)$ to $P(i+1)$. Let us choose $\eta_1,\mu \ll \eta_2 \ll 1$ so that each pair $(P(i),P(i+1))$ is $(\eta_2,1)$-super-regular in $G'_3$. Let $F'$ be the $4$-partite graph with vertex classes $A=P(1),B=P(2),C=P(3),D=P(4)$, where the $4$ bipartite graphs induced by $(P(i),P(i+1))$ are all complete. Clearly $F'$ contains a Hamilton cycle. We may apply \cref{lem:blowup} with $\Delta=2$ and $k=4$ to find a Hamilton cycle $C_{Ham}$ in $G'_3$, which corresponds to a directed Hamilton cycle in $G_3$ and thus, in turn, to a directed Hamilton cycle in $G$.

    \textit{Step~$5$}. Prove `moreover' part.
    Now, we deduce the `moreover' part from the first part of claim~\ref{cl:4}. Similarly as in the
    first part, the approach is to find a suitable path $P$ containing $x$ which we can contract into a single vertex so that the resulting oriented graph still satisfies $|B| = |D|$ and now all of its vertices are acceptable. Choose $x^-$ and $x^+$ so that $x^-x$ and $xx^+$ are acceptable edges. Since $x^-$ is acceptable, it has a cyclic inneighbor $x^{--}$ so that the edge $x^{--}x^-$ is acceptable. Let $P(i)$ be the class which contains $x^{--}$. Let $P'$ be a path of length at most $3$ starting at $x^{+}$ and ending at a cyclic vertex in $P(i)$ so that successive vertices lie in successive classes. Let $P\coloneqq x^{--}x^{-}xx^{+}P'$. Then $P$ is acceptable and thus $BD$-balanced. Let $H$ be the oriented graph obtained from $G$ by contracting $P$. Then in $H$ we still have $|B| = |D|$. All vertices that were previously acceptable/cyclic are still so (possibly with a larger error term $O(\eta_1)$. Since $x^{--}$ and the terminal vertex of $P$ are both cyclic (and thus  acceptable), this means that the new vertex resulting from the contraction of $P$ is still acceptable. Therefore we can apply the first part of the claim to obtain a Hamilton cycle in $H$ that clearly corresponds to one in the original oriented graph $G$.
\end{proof}

Since we are assuming that there is no Hamilton cycle, Claim~\ref{cl:4} gives $|B| \ne |D|$. Let us further assume that $|B| > |D|$. Although one cannot deduce a symmetry of $B$ and $D$ by the definition of extremity. In fact, the following proof will no longer use the fact that $e(D) = o(n^2)$. We will show that the set of non-acceptable edges of $A$, $C$, and the larger one among $B$ and $D$ can always be precisely characterised under the assumption that $G$ is not Hamiltonian.

We now say that a vertex is \emph{good} if it is acceptable, and also has one of the properties
$$A : B_{<\mu}C_{<\mu}, ~~~~B: A^{<\mu}B_{<\mu}^{<\mu}C_{<\mu}, ~~~~C: A^{<\mu}B^{<\mu} \text{~~~~or~~~~} D :$$
Otherwise, we say the vertex is \emph{bad}. Notice that the last option means that every acceptable vertex in $D$ is automatically good, and a
cyclic vertex is not necessarily good.

\begin{claim}\label{cl:6}
For each of the properties, $A : C_{>\mu}$, $A : B_{>\mu}$, $C : A^{>\mu}$, $C: B^{>\mu}$,
there are fewer than $|B|-|D|$ vertices with that property. In addition, by reassigning at most $O(\eta_1)n$ vertices, we can arrange that every vertex is good.
\end{claim}

Let $M$ be a maximum matching consisting of edges in $E(B,A)\cup E(B)\cup E(C,A) \cup E(C,B)$.
Say that $M\cap E(B,A)$ matches $B_A \subseteq B$ with $A_B\subseteq A$, that $M \cap E(B)$ is a matching on $B_B\subseteq B$, that $M \cap E(C,A)$ matches $C_A \subseteq C$ with $A_C \subseteq A$ and that $M \cap E(C,B)$ matches $C_B \subseteq C$ with $B_C \subseteq B$. 
Note that $e(M)=|A_B|+|A_C|+|B_B|/2+|B_C|$.

We must have $e(M)<|B|-|D|$; otherwise, there would exist $t\coloneqq|B|-|D|$ edges
$v^-_iv_i$ for $i\in [t]$ in $M$. By an argument similar to that in Claim~\ref{cl:6}, these $t$ edges $v^-_iv_i$ could be extended to $t$ vertex-disjoint paths $P_1,\ldots, P_t$, where each $P_i$ contains $v^-_iv_i$, starts and ends at cyclic vertices $b^-_i,b^+_i$ in the same class, and uses two more vertices from $B$ than from $D$.
As a result, we would obtain a Hamilton cycle as in Claim~\ref{cl:4}, yielding a contradiction.

\begin{claim}\label{cl:7}
$e(M)=0$ and $|B|-|D|=1$.
\end{claim}
\begin{proof}
We will first show that $e(M)=0$. Suppose, for contradiction, that $e(M)\ge 1$. 
Define
$A'\coloneqq A\setminus (A_B \cup A_C)$, $B'\coloneqq B \setminus (B_A \cup B_B\cup B_C)$, and $C'\coloneqq C \setminus (C_A \cup C_B)$. By the maximality of $M$, there are no edges from $B'\cup C'$ to $A'$ or from $B'\cup C'$ to $B'$. Since all vertices are good, it follows that
\begin{align}
 e(C,A)&\le e(C_A\cup C_B, A)+e(C,A_B\cup A_C)\nonumber\\
 &\le |C_A\cup C_B|\cdot(\mu+O(\eta_1))|A|+|A_B\cup A_C|\cdot(\mu+O(\eta_1))|C|\nonumber\\
 &\le e(M)(\mu+O(\eta_1))(|A|+|C|)\nonumber\\
&\le 2\mu n \cdot e(M).\label{equ:CA-1}
\end{align}
Similarly, we can obtain that 
\begin{align}
 e(B,A)\le e(B_A\cup B_B\cup B_C, A)+e(B,A_B\cup A_C)
\le 4\mu n\cdot e(M) \label{equ:BA-1}
\end{align}
and
\begin{align}
 e(C,B)\le e(C_A\cup C_B, B)+e(C,B_A\cup B_B\cup B_C)\le  4\mu n\cdot e(M).\label{equ:CB-1}
\end{align}
By the maximality of $M$, every edge in $B$ is incident to a vertex in $B_A\cup B_B\cup B_C$.
Moreover, since every vertex in $B$ is good, it follows that
\begin{align}
e(B)\le |B_A\cup B_B\cup B_C|\cdot 2(\mu+O(\eta_1)|B|\le 4\mu n \cdot e(M).\label{equ:B-1}
\end{align}

Consider all the ordered triples $(x,y,z)$ with $x\in C$, $y\in B$, and $z\in A$ such that $xy\notin E(G)$ and $yz\notin E(G)$. The number of such triples that do not satisfy the adjacency conditions is at most $e(B,A)|C|-e(C,B)|A|$. By Ore-degree condition we have
\begin{align}
    \sum_{\substack{x\in C,y\in B,z\in A\\xy,yz\notin E(G)}}\deg^+(x)+\deg(y)+\deg^-(z)&\ge \frac{3n-3}{2}\Big(|A||B||C|-e(B,A)|C|-e(C,B)|A|\Big)\nonumber\\
    &=\frac{3n-3}{2}|A||B||C|\left(1-O(\mu)e(M)/|A|\right).\label{equ:e(M)l}
\end{align}
On the other hand, by~(\ref{equ:CA-1})-(\ref{equ:B-1}), we obtain
\begin{align}
    \text{LHS of (\ref{equ:e(M)l})}& \le |A||B|\sum_{\substack{x\in C}}\deg^+(x)+|A||C|\sum_{\substack{y\in B}}\deg(y)+|B||C|\sum_{\substack{z\in A}}\deg^-(z)\nonumber\\
    &\le |A||B|\left(\binom{|C|}{2}+e(C,D)+e(C,A)+e(C,B)\right)\nonumber\\
    &~~+|A||C|\Big(  |B||A|+|B||C|+|B||D|+2e(B)\Big)\nonumber\\
    &~~+|B||C|\left(\binom{|A|}{2}+e(B,A)+e(C,A)+e(D,A)\right)\nonumber\\
    &\le|A||B|\Big(|C|^2/2+|C||D|+6\mu n\cdot e(M)\Big)\nonumber\\&~~+|A||C|\Big(|A||B|+|B||C|+|B||D|+8\mu n\cdot e(M)\Big)\nonumber\\&~~+|B||C|\Big(|A|^{2}/2+|A||D|+6\mu n\cdot e(M)\Big)\nonumber\\
    &= 3/2\cdot|A||B||C|\left(n-|B|+|D|\right)+O(\sqrt{\mu })|A||B||C| \cdot e(M).\label{equ:e(M)u} 
\end{align}
It follows from (\ref{equ:e(M)l}) and (\ref{equ:e(M)u}) that
\begin{align*}
    \frac{3n-3}{2}&\le \frac{3/2\cdot \left(n-|B|+|D|\right)+O(\sqrt{\mu})e(M)}{1-O(\mu)e(M)/|A|}\\
    &=  3/2\cdot \left(n-|B|+|D|\right)+  O(\sqrt{\mu})e(M).
\end{align*}
Recall that $e(M)< |B|-|D|$. We obtain that
\begin{align*}
e(M)+1\le |B|-|D|\le 1+O(\sqrt{\mu})e(M),
\end{align*}
which contradicts our assumption that $e(M) \ge 1$. Therefore, we have $e(M) = 0$, and consequently, $|B|-|D|=1$.
\end{proof}

Now we are ready to prove \cref{thm:extremal} under the assumption that $A \ge \sqrt{\mu} \cdot n$. Since $e(M)=0$, we have $xy,yz\notin E(G)$ for every triple $(x,y,z)$ with $x\in C,y\in B$, and $z\in A$. So we may choose a triple $(x,y,z)$ with $x\in C,y\in B$, and $z\in A$ such that
\begin{align*}
    \frac{3n-3}{2}&\le \deg^+(x)+\deg(y)+\deg^-(z)\\
    &\le \left(\frac{|C|-1}{2} + |D|\right) + (|A|+|C|+|D|) + \left(\frac{|A|-1}{2}+ |D|\right)\\
    &= \frac{3(|A|+|C|+2|D|)}{2}-1\\
    &= \frac{3(n-1)}{2}-1,
\end{align*}
which is a contradiction.

\subsection{When $|A| <\sqrt{\mu}\cdot n$}
The key point of this situation is that one can no longer guarantee the super-regularity of $E(A,B)$ and $E(D,A)$ whenever $A$ is relatively small. Thus, we show, in \cref{cl:4'}, that a corresponding Hamilton cycle can be found in a $3$-partite graph. Accordingly, we need to carefully modify the concepts of \emph{cyclic}, \emph{acceptable}, and \emph{good} vertices relative to $A$. The rest of the proof follows from the above section by adapting to the modified concepts. Analogously, we focus on the proofs of \cref{cl:4'} and \ref{cl:7'}. The proofs of other auxiliary claims will be omitted and put in Appendix.

Given a vertex $x\in G$, we will use the compact notation $x$ has property $W:B_{*}^{*}C_{*}^{*}D_{*}^{*}$ as follows. The specific meaning of it is essentially the same as in Section~\ref{subsec:7.1}, except for the error term, in which we use 
$O(\sqrt{\mu})n$ instead of $O(\eta_1)n$.
For example, to say $x$ has property $A: B^{>1}C^{<1/2}_{<1/3}$ means that $x\in A$, $|N^+(x)\cap B|>|B|-O(\sqrt{\mu})n$, $|N^+(x)\cap C|<|C|/2+O(\sqrt{\mu})n$ and $|N^-(x)\cap C|<|C|/3+O(\sqrt{\mu})n$.

It is worth noting that we use the notation $B\coloneqq P(1)$, $C\coloneqq P(2)$, $D\coloneqq P(3)$ in this subsection.
We say that a vertex $x\in V(G)\setminus A$ is \emph{cyclic} if it satisfies $$P(i): P(i+1)^{>1}P(i-1)_{>1}$$
for some $i\le 3$ (counting modulo $3$), and every vertex $x\in A$ is not cyclic. Since $G$ is $\eta$-extremal, the number of non-cyclic vertices is at most $O(\sqrt{\mu})n$. We say that a vertex is \emph{acceptable} if it has one of the following properties:
\begin{itemize}
        \item $A: B^{>1-\sqrt{\lambda}}C_{<\lambda}^{<\lambda}D_{\ge1-\sqrt{\lambda}}$;
        \item $B: C^{>\lambda}D_{>\lambda}$, $B: D_{>\lambda}^{>\lambda}$;
        \item $C: B_{>\lambda}D^{>\lambda}$, $C: B_{>\lambda}C^{>\lambda}$, $C: C_{>\lambda}D^{>\lambda}$, $C: C_{>\lambda}^{>\lambda}$;
        \item $D: B^{>\lambda}C_{>\lambda}$, $D: B_{>\lambda}^{>\lambda}$;
\end{itemize}

\begin{claim}\label{cl:3'}
By reassigning non-cyclic vertices to $A,B,C,$ or $D$, we can arrange that every vertex of $G$ is acceptable. 
\end{claim}

The following claim can be used to find Hamilton cycle in this case.
\begin{claim}\label{cl:4'}
If $|B|=|D|$, $|A| = O(\sqrt{\mu}) n$, and every vertex is acceptable, then $D$ has a Hamilton cycle. Moreover, the assertion remains valid under a slightly relaxed condition: $|B|=|D|$, and there exists a vertex $x$ such that all vertices except $x$ are acceptable, there is at least one acceptable edge directed toward $x$ and one acceptable edge directed away from $x$.
\end{claim}

\begin{proof}
We will show that, with some slight modifications, one can find a Hamilton cycle as what have been done in the proof of \ref{cl:4}.

\textit{Step~$1$}. Contract some suitable paths into vertices such that the resulting oriented graph consists entirely of cyclic vertices.
Note that in this case $B\coloneqq P(1)$, $C\coloneqq P(2)$, $D\coloneqq P(3)$, and every vertex of $A$ is non-cyclic.
Similarly to Step 1 of Claim~\ref{cl:4}, we can find suitable paths which together contain all the non-cyclic vertices. Then we will contract these paths into vertices so that the resulting oriented graph
$G_1$ consists entirely of cyclic vertices. 
Moreover, the sets $B, C, D$ in $G_1$ still satisfy $|B|=|D|$, $A=\emptyset$, and we still have that the sizes of the other pairs of sets differ by at most $O(\sqrt{\mu})n$.

\textit{Step~$2$}. Contract some suitable paths so that $|B|=|C|=|D|$. This can be done by a similar argument as the proof of \cref{cl:4}. Let $G_2$ be the resulting oriented graph.

\textit{Step~$3$}. Find a Hamilton cycle in $G_2$. Clearly every vertex of $G_2$ is cyclic. Let $G'_2$ be the underlying graph corresponding to the set of edges oriented from $P(i)$ to $P(i+1)$. By choosing a propriate $\mu'$ so that each pair $(P(i),P(i+1))$ is $(\mu',1)$-super-regular in $G'_2$, we may apply \cref{lem:blowup} to find a Hamilton cycle in $G'_2$, which corresponds to a directed Hamilton cycle in $G_2$ and thus, in turn, to a directed Hamilton cycle in $G$.

The `moreover' part immediately follows from a similar proof of \cref{cl:4}.
\end{proof}

Let us assume that $|B| \ge |D|$. Since we are assuming that there is no Hamilton cycle, Claim~\ref{cl:4'} gives $|B| > |D|$.  We now define that a vertex is \emph{good} if it is acceptable, and also has one of the properties
$$A : B_{<\sqrt{\lambda}}C_{<\lambda}, ~~~~B: B_{<\lambda}^{<\lambda}C_{<\lambda},~~~~ C: B^{<\lambda} \text{~~~~or~~~~} D :$$
Otherwise, we say the vertex is \emph{bad}. Note that the last options mean that every acceptable vertex in $D$ is automatically good and a
cyclic vertex is not necessarily good.

\begin{claim}\label{cl:6'}
If $|A| = O(\sqrt{\mu}) n$, then by reassignment at most $O(\sqrt{\mu})n$ vertices we can arrange that every vertex is good.
\end{claim}

Let $M$ be a maximum matching consisting of edges in $E(B,A)\cup E(B)\cup E(C,A) \cup E(C,B)$ with the property that $|V(M)\cap A|$ is as large as possible.
Say that $M\cap E(B,A)$ matches $B_A \subseteq B$ with $A_B\subseteq A$, that $M \cap E(B)$ is a matching on $B_B\subseteq B$, that $M \cap E(C,A)$ matches $C_A \subseteq C$ with $A_C \subseteq A$ and that $M \cap E(C,B)$ matches $C_B \subseteq C$ with $B_C \subseteq B$. 
Note that $e(M)=|A_B|+|A_C|+|B_B|/2+|B_C|$.
Here we also must have $e(M)<|B|-|D|$ by the same argument as that in the case $|A|=\Omega(\sqrt{\mu}n)$ together with the fact that each vertex of $A$ satisfies the property $A:B^{>1-\sqrt{\lambda}}C_{<\lambda}^{<\lambda}D_{>1-\sqrt{\lambda}}$.

\begin{claim}\label{cl:7'}
$e(M)=0$ and $A = \emptyset$.
\end{claim}
\begin{proof}

Assume to the contrary that $e(M)\ge 1$. 
Define
$A'\coloneqq A\setminus (A_B \cup A_C)$, $B'\coloneqq B \setminus (B_A \cup B_B\cup B_C)$ and $C'\coloneqq C \setminus (C_A \cup C_B)$. By the maximality of $M$, there are no edges from $B'\cup C'$ to $A'$ or from $B'\cup C'$ to $B'$. Since $|A|=O(\sqrt{\mu})n$ and all vertices are good, it follows that
\begin{align}
 e(C,A)&\le e(C_A\cup C_B, A)+e(C,A_B\cup A_C)\nonumber\\
 &\le |C_A\cup C_B|\cdot|A|+|A_B\cup A_C|\cdot(\lambda+O(\sqrt{\mu}))|C|\nonumber\\
 &\le e(M)|A|+e(M)(\lambda+O(\sqrt{\mu}))|C|\nonumber\\
 &\le \lambda n \cdot e(M).\label{equ:CA}
\end{align}
Similarly, we can obtain that 
\begin{align}
 e(B,A)\le e(B_A\cup B_B\cup B_C, A)+e(B,A_B\cup A_C)\le \lambda n \cdot e(M)\label{equ:BA}
\end{align}
and
\begin{align}
 e(C,B)\le e(C_A\cup C_B, B)+e(C,B_A\cup B_B\cup B_C)\le 2\lambda n \cdot e(M).\label{equ:CB}
\end{align}
By the maximality of $M$, every edge in $B$ is incident to a vertex in $B_A\cup B_B\cup B_C$.
Moreover, since every vertex in $B$ is good, it follows that
\begin{align}
e(B)\le |B_A\cup B_B\cup B_C|\cdot 2(\lambda+O(\sqrt{\mu}))|B|\le 2 \lambda n \cdot e(M).\label{equ:B}
\end{align}

Now we give more precise bounds of the size of $A,B,C,D$ with respect to $e(M)$. First we compute the upper bound of $|B|$. There are exactly $|B|(|B|-1)-e(B)$ pairs $(x,y)$ with $x\neq y$ such that $x,y\in B$ and $xy\notin E(G)$. So we have
\begin{align}
    \frac{3n-3}{4} \left(|B|(|B|-1)-e(B)\right) &\le \sum_{\substack{x,y\in B,x\neq y\\xy\notin E(G)}}\deg^+(x)+\deg^-(y)\nonumber \\
    &\le (|B|-1)\sum_{x\in B} \deg(x)\nonumber \\
    &\le |B|(|B|-1)(|A| + |C| + |D|) + 2|B|e(B)).\nonumber
\end{align}
Hence
\begin{align}
    \frac{3n-3}{4} &\le \frac{|B|(|B|-1)(|A| + |C| + |D|) + 2|B|e(B))}{|B|(|B|-1)-e(B)}\nonumber \\
    &\le |A| + |C| + |D| + \frac{e(B)(|A| + 2|B| + |C| + |D|)}{|B|(|B|-1)-e(B)}\nonumber \\
    &\le |A| + |C| + |D| + O(\lambda) e(M).\nonumber
\end{align}
It follows that 
\begin{align}
    |B| &\le \frac{n+3}{4} + O(\lambda)e(M).\label{equ:eBu}
\end{align}

Then we compute the upper bound of $|C|$.
Clearly there are $|B||C|- e(C,B)$ pairs $(x,y)$ with $x\in C$, $y\in B$, and $xy\notin E(G)$.
So we have
\begin{align*}
    \frac{3n-3}{4} (|B||C|- e(C,B)) &\le \sum_{\substack{x\in C,y\in B,\\xy\notin E(G)}}\deg^+(x)+\deg^-(y)\\
    &\le |B| \sum_{x\in C} \deg^+(x) + |C| \sum_{y\in B} \deg^-(y).
\end{align*}
Similarly we can get that
\[\frac{3n-3}{4} \le |A| + \frac{|C|}{2} + 2|D| + O(\lambda) e(M).\]
Using that $e(M) \le |B| - |D| - 1$ we have
\begin{align}
    |C| &\le \frac{n-1}{2} - 2e(M) + O(\lambda)e(M).\label{equ:eCu}
\end{align}
Together with (\ref{equ:eBu}) we can further get that
\begin{align}
    |A| \ge (3 -O(\lambda))e(M).\label{equ:lowerA} 
\end{align}
which implies that $e(M)=O(\sqrt{\mu})n$.

Our aim is to find exactly a vertex $u\in C$ and a vertex $v\in A$ such that $uv\notin E(G)$, and moreover, $\deg^+(u) + \deg^-(v)$ does not meet Ore-degree condition, which gives a contradiction.

Let us first find a vertex in $C\setminus (C_A\cup C_B)$ whose outdegree is as small as possible. For a vertex $x\in C_A\cup C_B$, we may consider all $y\in B$ such that $xy\notin E(G)$. ByOre-degree condition we have
\begin{align*}
    \frac{3n-3}{4} (|B| - \deg^+_B(x)) &\le \sum_{y\in B,xy\notin E(G)} \deg^+(x) + \deg^-(y)\\
    &\le (|B| - \deg^+_B(x)) (\deg^+(x) + |A| + |D|) + e(B) + e(C,B).
\end{align*}
Notice that $x$ is a good vertex, we have $\deg^+_B(x) \le O(\lambda) n$. And hence
\begin{align}
    \deg^+(x) \ge \frac{3n-3}{4} - |A| - |D| - O(\lambda) e(M).\label{eq:CinMoutdeg}
\end{align}

On the other hand, we can give an upper bound to the sum of the outdegree of all vertices in $C$. We have
\begin{align}
    \sum_{x\in C} \deg^+(x) \le \binom{|C|}{2} + |C||D| + e(C,A) + e(C,B).\label{eq:Cavd}
\end{align}
Now combining (\ref{eq:CinMoutdeg}) and (\ref{eq:Cavd}), there exists a vertex $u\in C\setminus (C_A\cup C_B)$ whose outdegree satisfies that
\begin{align}
    \deg^+(u) &\le \frac{\sum_{x\in C} \deg^+(x) - \sum_{x\in C_A\cup C_B} \deg^+(x)}{|C\setminus (C_A\cup C_B)|}\nonumber \\
    &\le \frac{\binom{|C|}{2} + |C||D| + e(C,A) + e(C,B) - e(M) (\frac{3n-3}{4} - |A| - |D| - O(\lambda) e(M))}{|C| - e(M)}\nonumber \\
    &\le \frac{|C|-1}{2} + |D| + \frac{e(M) \left(\frac{|C|-1}{2} + |D| - \frac{3n-3}{4} + |A| + |D| + O(\lambda) n\right)}{|C| - e(M)}\nonumber \\
    &\le \frac{|C|-1}{2} + |D| + O(\lambda) e(M).\label{eq:CnotinMoutdeg}
\end{align}

It remains to find a vertex in $A$. Let $M_1$ be the set of all edges $xy\in E(M)$ such that $y\in A$ and $N^-_{B\cup C}(y) \subseteq V(M)$, that is, every inneighbor of $y$ in $B\cup C$ must belong to $M$. Let $M_2$ be the set of all edges $xy\in E(M)$ such that $y\in A$ and $N^-_{B\cup C}(y) \setminus V(M) \neq \emptyset$, and let $M_3\coloneqq M \setminus (M_1\cup M_2)$. By the maximality of $M$ and the definition of $M_1$, for every vertex $x\in A\setminus V(M_2)$ we have $ux\notin E(G)$.

We still need some additional assumptions and definitions in order to find a suitable vertex in $A$. Let $M$ be such that the size of $A \cap V(M)$ is maximum, i.e. $e(M_1)+e(M_2)$ is as large as possible, subject to that $M$ is a maximum matching. Let $M_{11}$ be the set of all edges $xy\in E(M_1)$ such that $N^+_A(x)\setminus V(M) \neq \emptyset$, and let $M_{12}\coloneqq M_1\setminus M_{11}$. Write $A_0\coloneqq A\setminus V(M)$ and $A_{11}\coloneqq A_0 \cup (A\cap V(M_{11}))$. We then show that there is a vertex $v$ in $A_0$ or $A_{11}$ that satisfies our requirement.

For every vertex $z\in A_0$, the inneighbors of $z$ must be contained in $A_0\cup D\cup V(M)$. In fact, not every vertex in $M$ can be an inneighbor of $z$. For $xy\in E(M_{12})$, we have $xz\notin E(G)$ by the definition of $M_{12}$. For $xy\in E(M_2)$, we claim that $xz\notin E(G)$, since otherwise we can find $y'\in N^-_{B\cup C}(y) \setminus V(M)$ such that $y'yxz$ is an augmenting path for $M$, which contradicts the maximality of $M$. For $xy\in E(M_3)$, we claim that $xz,yz\notin E(G)$, since otherwise we can remove $xy$ and add $xz$ or $yz$ to $M$, which contradicts the maximality of $A\cap V(M)$. We have
\[\sum_{z\in A_0} \deg^-(z) \le \binom{|A_0|}{2} + |A_0||D| + |A_0| (2e(M_{11}) + e(M_{12}) + e(M_2)).\]
Hence there exists $v_1\in A_0$ such that
\begin{align}
    \deg^-(v_1) \le \frac{|A|-1}{2} + |D| + \frac{3}{2} e(M_{11}) + \frac{1}{2} e(M_{12}) + \frac{1}{2} e(M_2).\label{eq:mindegA0}
\end{align}

On the other hand, let us consider the indegree of $w$ for which $w'w\in E(M_{11})$. For $xy\in M_2$, we claim that $xw\notin E(G)$, since otherwise we can find $w''\in N^+_A(w')$ and $y'\in N^-_{B\cup C}(y) \setminus V(M)$ such that $y'yxww'w''$ is an angmenting path for $M$, a contradiction. Similarly for $xy\in E(M_3)$, we have $xw,yw\notin E(G)$ by the maximality of $A\cap V(M)$. So we have
\[\sum_{w\in A_{11}} \deg^-(w) \le \binom{|A_{11}|}{2} + |A_{11}||D| + |A_{11}| (e(M_{11}) + 2e(M_{12}) + e(M_2)).\]
Hence there exists $v_2\in A_{11}$ such that
\begin{align}
    \deg^-(v_2) \le \frac{|A|-1}{2} + |D| + e(M_{11}) + \frac{3}{2} e(M_{12}) + \frac{1}{2} e(M_2).\label{eq:mindegA11}
\end{align}

Now combining (\ref{eq:mindegA0}) and (\ref{eq:mindegA11}), and noting that $e(M)\ge e(M_{11}) + e(M_{12}) + e(M_2)$, we can find $v\in \{v_1,v_2\}$ such that
\begin{align}
    \deg^-(v) \le \frac{|A|-1}{2} + |D| + \frac{5}{4} e(M).\label{eq:AnotinM2indeg}
\end{align}

So far we have found two vertices $u\in C$ and $v\in A$ such that $uv\notin E(G)$, where $\deg^+(u)$ and $\deg^-(v)$ follow from (\ref{eq:CnotinMoutdeg}) and (\ref{eq:AnotinM2indeg}), respectively. We have
\begin{align*}
    \deg^+(u) + \deg^-(v) &\le \frac{|A|}{2} + \frac{|C|}{2} + 2|D| - 1 + \left( \frac{5}{4} + O(\lambda) \right) e(M)\\
    &\le \frac{|A|}{2} + \frac{|C|}{2} + \frac{3}{2} (|B| - e(M) - 1) + \frac{|D|}{2} - 1 + \left( \frac{5}{4} + O(\lambda) \right) e(M)\\
    &\le \frac{n}{2} + |B| - \frac{5}{2} - \left( \frac{1}{4} - O(\lambda) \right) e(M)\\
    &\le \frac{3n-7}{4} - \left( \frac{1}{4} - O(\lambda) \right) e(M),
\end{align*}
which contradictsOre-degree condition. Note that the above proof still holds when $e(M)=0$ and $A\neq \emptyset$. Therefore, we have $e(M) = 0$ and $A = \emptyset$.
\end{proof}

We are now already to prove Theorem~\ref{thm:extremal} under the asumption that $|A| < \sqrt{\mu}\cdot n$. Since $e(M) = 0$, we have $xy,yz\notin E(G)$ for every triple $(x,y,z)$ with $x\in C$ and $y,z\in B$. So we may choose a triple such that
\begin{align*}
    \frac{3n-3}{2} &\le \deg^+(x) + \deg(y) + \deg^-(z)\\
    &\le \left( \frac{|C|-1}{2} + |D| \right) + (|C| + |D|) + |D|\\
    &\le \frac{3}{2} |C| + \frac{3}{2} (|B| - 1) + \frac{3}{2} |D| - \frac{1}{2}\\
    &=\frac{3n-4}{2},
\end{align*}
a contradiction.
\section{Concluding remarks}
It is interesting to seek more applications of the technique we introduced here, especially the absorbing method dealing with the non-extremal case.
Given an oriented graph $G$, for each $v \in V(G)$, let $\delta^*(G) \coloneqq\delta(G) + \delta^+(G) + \delta^-(G)$. H\"{a}ggkvist~\cite{haggkvist1993hamilton} conjectured the following degree sum condition for Hamilton cycles in oriented graphs:
\begin{conj}
Every oriented graph $G$ on $n$ vertices with $\delta^*(G) > (3n-3)/2$ contains a Hamilton cycle.
\end{conj}

This conjecture has been confirmed asymptotically in~\cite{kelly2008dirac}.
We believe our methods could be helpful in proving a stable version of the above conjecture, offering hope for a full resolution of the conjecture for large graphs. 

\bibliographystyle{abbrv}
\bibliography{ref}

\begin{appendix}
\section*{Appendix: supplementary proofs}
\label{sec:appendix}
For completeness, we give the proof of Claims~\ref{cl:3},~\ref{cl:6},~\ref{cl:3'},~\ref{cl:6'} in this Appendix.
\begin{proof}[Proof of Claim~\ref{cl:3}]
     For any vertex $x$, \cref{le:delta0} guarantees that one can always find $P(i)$ and $P(j)$ such that $|N^+(x)\cap P(i)| \ge n/32$ and $|N^-(x)\cap P(j)| \ge n/32$ for some $1\le i,j\le 4$. There are at most $16$ possible positions of $P(i)P(j)$, each of which corresponds to one case in the definition of acceptable vertices. Hence we can put $x$ into the correct position such that it becomes acceptable. Since there are at most $O(\eta_1)n$ non-cyclic vertices, the reassignment operations do not affect the properties of all cyclic vertices by increasing the hidden constant in the $O(\eta_1)n$-notation if necessary.
\end{proof}

\begin{proof}[Proof of Claim~\ref{cl:6}]
Firstly, we prove that the number of bad vertices in $A$ and $C$ is at most $O(\eta_1)n$. To be specific,
for each of the properties, $A : C_{>\mu}$, $A : B_{>\mu}$, $C : A^{>\mu}$, $C: B^{>\mu}$,
there are fewer than $|B|-|D|$ vertices with that property. Suppose to the contrary that, for example, we can find $|B|-|D| \eqqcolon t$ vertices $v_1, \ldots , v_t$ in $A$ having the property $A : C_{>\mu}$. Select distinct cyclic vertices $v^-_1 ,\ldots, v^-_t$ in $C$ with $v^-_i\in N^-(v_i)$ for $1 \le i\le t$. This can be achieved greedily, as $|B|-|D| = O(\eta_1)n$ and $\eta_1\ll \mu$. Since $v_1, \ldots, v_t$ are acceptable vertices in $A$ we can greedily select distinct cyclic vertices $v^+_i\in N^+(v_i)\cap (A\cup B)$. For each $i$ such that $v^+_i\in B$ let $b^+_i\coloneqq v^+_i$. For each $i$ such that $v^+_i\in A$ select $b^+_i$ to be a cyclic vertex in $N^+(v^+_i)\cap B$; we can ensure that $b^+_1, \ldots , b^+_t$ are distinct. Similarly we can select cyclic vertices $b^-_i\in N^-(v^-_i)\cap B$, which are distinct from each other and from $b^+_1, \ldots, b^+_t$. Thus we have constructed $t=|B|-|D|$ vertex disjoint paths $P_1, \ldots, P_t$, where $P_i$ starts at $b^-_i$ in $B$, goes through $C$, then it uses $1$ or $2$ vertices of $A$, and it ends at $b^+_i$ in $B$. Consider a new oriented graph $H$, obtained from $G$ by contracting these paths.
Note that the vertices of $H$ corresponding to the paths $P_i$ are acceptable and the analogues of $B$ and $D$ in $H$ now have equal size. Therefore we can apply Claim~\ref{cl:4} to find a Hamilton
cycle in $H$. This corresponds to a Hamilton cycle in $G$, which contradicts our assumption.
Therefore there are less than $|B|-|D|$ vertices with property $A:C_{>\mu}$.

The statement for property $C:A^{>\mu}$ follows by the same argument above, as we can again start by finding a matching of size $|B|-|D|$ consisting of edges directed from $C$ to $A$. The arguments for
the other two properties are also similar. For instance, if we have $|B|-|D|$ vertices $v_1, \ldots , v_t$ in $A$ having property $A : B_{>\mu}$, we can find a matching $b_1^-v_1,\ldots,b_t^-v_t$ of edges directed from $B$ to $A$ such that the $b_i^{-}$ are cyclic. We can extend this to $|B|-|D|$ vertex disjoint paths $P_1,\ldots, P_t$, where $P_i$ starts with the edge $b_i^-v_i$, it then either goes directly to a cyclic vertex in $B$ or uses one more vertex from $A$ before it ends at a cyclic vertex in $B$. Now we construct a new oriented graph $H$ similarly as before and find a Hamilton cycle in $H$ and then in $G$.

Secondly, we show that there are fewer than $|B|-|D|$ bad vertices in $B$. Suppose that there are at least $|B|-|D|$ bad vertices in $B$. The idea here is to prove we can reassign bad vertices of $B$ one by one until satisfying the condition of Claim~\ref{cl:4}, and then $G$ has a Hamilton cycle, which is a contradiction. The key of this proof is to deal with the circumstance that $|B| -|D|$ goes from $+1$ to $-1$ if a vertex $x$ is moved from $B$ to $D$ during this process by using the `moreover' part of Claim~\ref{cl:4}. We now first analyze the degree of bad vertices. Let $x$ be a bad vertex in $B$. 

\textbf{Case 1: $|N^+(x)\cap A| > \mu|A|$ or $|N^+(x)\cap B| > \mu|B|$.} If $|N^-(x) \cap C| > \mu|C|$ or $|N^-(x) \cap B| > \mu|B|$, then by reassigning $x$ to $D$ and we have decreased $|B|-|D|$ by $2$. Otherwise, we have $|N^-(x) \cap C| \le \mu|C|$ and $|N^-(x) \cap B|\le \mu|B|$, then we must have $|N^-(x) \cap A| \ge \mu|A|$ or $|N^-(x) \cap D| \ge \mu|D|$ by Proposition~\ref{le:delta0}. We now reassign $x$ to $A$ and $|B|-|D|$ will be decreased by $1$.

\textbf{Case 2: $|N^+(x)\cap A| \le \mu|A|$ and $|N^+(x)\cap B| \le \mu|B|$.} This implies that $|N^+(x) \cap C| \ge \mu|C|$ or $|N^+(x) \cap D| \ge \mu|D|$ by Proposition~\ref{le:delta0}.
Moreover, since $x$ is a bad vertex in $B$ we have $|N^-(x)\cap B|\ge \mu|B|$ or $|N^-(x)\cap C|\ge \mu|C|$.
Then by reassigning $x$ to $C$ and we have decreased $|B|-|D|$ by $1$.

Note that all vertices after the operations above are still acceptable, and we have decreased $|B| - |D|$ by $1$ or $2$ for each reassignment.
Since the number of bad vertices in $B$ is at least $|B|-|D|$, during the process we must have either $|B|$ and $|D|$ become equal or $|B|- |D|$ goes from $+1$ to $-1$ if a vertex $x$ is moved from $B$ to $D$.
In this case we claim that we can put $x$ into $A$ or $C$ to achieve the following property: every vertex except $x$ is acceptable, there is at least one acceptable edge going into $x$ and at least one acceptable edge coming out of $x$.
Therefore we can apply Claim~\ref{cl:4} to find a Hamilton cycle in $G$, which would be a contradiction. 

It remains to show that, when $|B|-|D|$ goes from $+1$ to $-1$ if we move $x$ from $B$ to $D$, we can put $x$ into $A$ or $C$ to achieve the `moreover' part of Claim~\ref{cl:4}. 
To see this, note that $x$ is acceptable even though $x$ is a bad vertex of $B$, this implies that $N^-(x) \cap A\neq\emptyset$ or $N^-(x)\cap D\neq\emptyset$ by the definition of acceptable vertices. 
If $N^+(x)\cap A\neq\emptyset$ or $N^+(x) \cap B\neq\emptyset$ we can put $x$ into $A$. 
Otherwise $N^+(x) \cap A=\emptyset$ and $N^+(x)\cap B=\emptyset$.
So we have $N^+(x) \cap C\neq\emptyset$ or $N^+(x) \cap D\neq\emptyset$ by Proposition~\ref{le:delta0}.
Furthermore, since $x$ is a bad vertex in $B$, we have $N^-(x)\cap B\neq \emptyset$ or $N^-(x)\cap C\neq \emptyset$.
Then we can put $x$ into $C$, as required.

Now we are ready to show that at most $O(\eta_1)n$ vertices can be arranged such that every vertex is good. Notice that all vertices in $D$ are good, so there are at most $O(\eta_1)n$ bad vertices in $G$. Now we reassign all bad vertices as follows. 

Let $x\in A\cup B\cup C$ be a bad vertex. 

\textbf{Case 1: $x\in B\cup C$ with $|N^+(x)\cap A| > \mu|A|$ or $|N^+(x)\cap B| >\mu|B|$.} If $|N^-(x) \cap C| > \mu|C|$ or $|N^-(x) \cap B| > \mu|B|$, then we can reassign $x$ to $D$. Otherwise, we have $|N^-(x) \cap C| \le \mu|C|$ and $|N^-(x) \cap B|\le \mu|B|$, then we must have $|N^-(x) \cap A| \ge \mu|A|$ or $|N^-(x) \cap D| \ge \mu|D|$ by Proposition~\ref{le:delta0}. We now reassign $x$ to $A$.

\textbf{Case 2: $x\in A\cup B$ with $|N^-(x)\cap B|\ge \mu|B|$ or $|N^-(x)\cap C|\ge \mu|C|$.}  If $|N^+(x)\cap A|\ge \mu|A|$ or $|N^+(x)\cap B|\ge \mu|B|$, then we reassign $x$ to $D$. Otherwise, we have $|N^+(x)\cap A|\le \mu|A|$ and $|N^+(x)\cap B|\le \mu|B|$. It implies that $|N^+(x) \cap C| \ge \mu|C|$ or $|N^+(x) \cap D| \ge \mu|D|$, and then we reassign $x$ to $C$.

We have thus proved the statement of this claim.
\end{proof}

\begin{proof}[Proof of Claim~\ref{cl:3'}]
    For any vertex $x$, \cref{le:delta0} guarantees that one can always find $P(i)$ and $P(j)$ such that $x$ has sufficiently many outneighbors in $P(i)$ and inneighbors in $P(j)$. There are $9$ possible positions of $P(i)P(j)$. Unless $P(i)=B, P(j)=D$ the vertex $x$ can be reassigned to one of $B,C,D$ such that it becomes an acceptable vertex. So we may assume that $x$ has the property $B^{>\lambda}D_{>\lambda}$ and is not acceptable with respect to $B,C,D$. This straightly implies that $x$ has that property $B_{<\lambda}C^{<\lambda}_{<\lambda}D^{<\lambda}$.
    
    Consider all vertices $y\in C$ with $yx\notin E(G)$. ByOre-degree condition we have
    \begin{align*}
        \frac{3n-3}{4}((1-\lambda)|C| - O(\sqrt{\mu})n) &\le \sum_{y\in C,yx\notin E(G)} \deg^+(y) + \deg^-(x)\\
        &\le |C| \deg^-(x) + e(C) + e(C,D) + e(C,A) + e(C,B)\\
        &\le |C| (\deg^-_B(x) + 3\lambda n) + \binom{|C|}{2} + |C||D| + O(\sqrt{\mu})n^2.
    \end{align*}
    Hence we have
    \[\deg^-_B(x) \ge \frac{3n}{4} - \frac{|C|}{2} - |D| - O(\lambda)n \ge \frac{n}{4} - O(\lambda)n,\]
    that is, $x$ has the property $B^{>1-\sqrt{\lambda}}$. Similarly we can prove that $x$ has the property $D_{>1-\sqrt{\lambda}}$. We finish the proof of this claim.
\end{proof}

\begin{proof}[Proof of Claim~\ref{cl:6'}]
By Claim~\ref{cl:3'}, each vertex in $A$ have the property $A:B^{>1-\sqrt{\lambda}}C_{<\lambda}^{<\lambda}D_{>1-\sqrt{\lambda}}$. It follows that $\deg_B^-(x)\le (\sqrt{\lambda}+O(\mu))|B|$, which means that all vertices in $A$ are good.

Firstly, we prove that the number of bad vertices in $C$ is at most $|B|-|D|$. To be specific, there are fewer than $|B|-|D|$ vertices with $C :B^{>\lambda}$. Let $t=|B|-|D|$. We find $t$ vertices $v_1, \ldots , v_t$ in $C$ having the property $C:B^{>\lambda}$. We select greedily distinct cyclic outneighbors of these $t$ vertices $v^+_1 ,\ldots, v^+_t$ in $B$ and let $b^+_i\coloneqq v^+_i$.
Since $v_1, \ldots, v_t$ are acceptable vertices in $C$, we can greedily select distinct cyclic vertices $v^-_i\in N^-(v_i)\cap (B\cup C)$. For each $i$ such that $v^-_i\in B$, let $b^-_i\coloneqq v^-_i$. For each $i$ such that $v^-_i\in C$, select $b^-_i$ to be a cyclic vertex in $N^-(v^-_i)\cap B$; we can ensure that $b^-_1, \ldots , b^-_t$ are distinct. Thus we have constructed $t=|B|-|D|$ vertex-disjoint paths $P_1, \ldots, P_t$, where $P_i$ starts at $b^-_i$ in $B$, goes through $C$, then it uses $1$ or $2$ vertices of $C$, and it ends at $b^+_i$ in $B$. Consider a new oriented graph $H$, obtained from $G$ by contracting these paths.
Note that the vertices of $H$ corresponding to the paths $P_i$ are acceptable and the analogues of $B$ and $D$ in $H$ now have equal size. Therefore we can apply Claim~\ref{cl:4'} to find a Hamilton
cycle in $H$. This corresponds to a Hamilton cycle in $G$, which contradicts our assumption.
Therefore there are less than $|B|-|D|$ vertices with property $C:B^{>\lambda}$.

Secondly, we show that the number of bad vertices in $B$ is at most $|B|-|D|$. Suppose that there are at least $|B|-|D|$ bad vertices in $B$. The idea here is to prove we can reassign bad vertices of $B$ one by one until satisfying the condition of Claim~\ref{cl:4'}, and then $G$ has a Hamilton cycle, which is a contradiction. The key of this proof is to deal with the circumstance that $|B| -|D|$ goes from $+1$ to $-1$ if a vertex $x$ is moved from $B$ to $D$ during this process by using the "moreover" part of Claim~\ref{cl:4'}. We now first analyze the degree of bad vertices. Let $x$ be a bad vertex in $B$. 

\textbf{Case 1: $|N^-(x)\cap B|\ge \lambda|B|$ or $|N^-(x)\cap C|\ge \lambda|C|$.} If $|N^+(x)\cap B|\ge \lambda|B|$, then we reassign $x$ to $D$ and $|B|-|D|$ will be decreased by 2; otherwise, $|N^+(x)\cap B|\le \lambda|B|$. Since $|A|=O(\sqrt{\mu})n$, we get that $|N^+(x) \cap C| \ge \lambda|C|$ or $|N^+(x) \cap D| \ge \lambda|D|$ by Proposition~\ref{le:delta0}. In this case we can reassign $x$ to $C$ and $|B|-|D|$ will be decreased by $1$.

\textbf{Case 2: $|N^-(x)\cap B|\le \lambda|B|$ and $|N^-(x)\cap C|\le \lambda|C|$.} Since $|A|=O(\sqrt{\mu})n$, we must have $|N^-(x) \cap D| \ge \lambda|D|$ by Proposition~\ref{le:delta0}.
Moreover, $|N^+(x) \cap B| > \lambda|B|$
since $x$ is a bad vertex in $B$.
Now we can reassign $x$ to $A$ and $|B|-|D|$ will be decreased by $1$. 
Note that all vertices after the operations above are still acceptable, and we have decreased $|B| - |D|$ by $1$ or $2$ for each reassignment.
Since the number of bad vertices in $B$ is at least $|B|-|D|$, during the process we must have either $|B|$ and $|D|$ become equal or $|B|- |D|$ goes from $+1$ to $-1$ if a vertex $x$ is moved from $B$ to $D$.
In this case we claim that we can put $x$ into $A$ or $C$ to achieve the following property: every vertex except $x$ is acceptable, there is at least one acceptable edge going into $x$ and at least one acceptable edge coming out of $x$.
Therefore we can apply Claim~\ref{cl:4'} to find a Hamilton cycle in $G$, which would be a contradiction. 

It remains to show that, when $|B|-|D|$ goes from $+1$ to $-1$ if we move $x$ from $B$ to $D$, we can put $x$ into $A$ or $C$ to achieve the `moreover' part of Claim~\ref{cl:4'}. 
To see this, note that $x$ is acceptable even though $x$ is a bad vertex of $B$, this implies that $N^-(x)\cap D\neq\emptyset$ by the definition of acceptable vertices. 
If $N^+(x) \cap B\neq\emptyset$ we can put $x$ into $A$; 
otherwise $N^+(x)\cap B=\emptyset$.
Together with $|A|=O(\sqrt{\mu})n$, we have $N^+(x) \cap C\neq\emptyset$ or $N^+(x) \cap D\neq\emptyset$ by Proposition~\ref{le:delta0}.
Furthermore, since $x$ is a bad vertex in $B$, we have $N^-(x)\cap B\neq \emptyset$ or $N^-(x)\cap C\neq \emptyset$.
Then we can put $x$ into $C$, as required.

Now we are ready to prove that we can arrange at most $O(\sqrt{\mu})n$ vertices such that every vertex is good.
Note that all vertices in $A\cup D$ are good, we know that there are at most $O(\sqrt{\mu})n$ bad vertices in $G$. 
While $|B|>|D|$ we reassign all bad vertices as follows. 
Suppose there is a bad vertex $x\in B\cup C$. 

\textbf{Case 1: $x\in C$.} Since $x$ is bad in $C$, we have $|N_B^+(x)|\ge \lambda|B|$. Also, $x$ is acceptable, it follows that $|N_B^-(x)|\ge \lambda|B|$ or $|N_C^-(x)|\ge \lambda|C|$. So we can reassign $x$ to $D$.

\textbf{Case 2: $x\in B$ with $|N^-(x)\cap B|\ge \lambda|B|$ or $|N^-(x)\cap C|\ge \lambda|C|$.} If $|N^+(x)\cap B|\ge \lambda|B|$, then we reassign $x$ to $D$; otherwise, $|N^+(x)\cap B|\le \lambda|B|$. Since $|A|=O(\sqrt{\mu})n$, we get that $|N^+(x) \cap C| \ge \lambda|C|$ or $|N^+(x) \cap D| \ge \lambda|D|$ by Proposition~\ref{le:delta0}. In this case we can reassign $x$ to $C$.

\textbf{Case 3: $x\in B$ with $|N^-(x)\cap B|\le \lambda|B|$ and $|N^-(x)\cap C|\le \lambda|C|$.} Since $|A|=O(\sqrt{\mu})n$, we must have $|N^-(x) \cap D| \ge \lambda|D|$ by Proposition~\ref{le:delta0}.
Moreover, $|N^+(x) \cap B| > \lambda|B|$
since $x$ is a bad vertex in $B$. It is clear that $x$ is good after reassigning $x$ to $A$. 

We finish the proof of this claim by reassigning all bad vertices at most $O(\sqrt{\mu})n$ times to make all vertices good.
\end{proof}

\end{appendix}

\end{document}